\pdfoutput=1
\documentclass[12pt,letter]{amsart}
\usepackage{xcolor}
\usepackage[all,color]{xy}
\usepackage{color}
\usepackage{graphicx,import}
\usepackage{amsmath,amssymb,amsthm,amscd,amsxtra,amsfonts,mathrsfs,graphicx,enumerate,bm,slashed,array,setspace,amsfonts,color,epstopdf,enumitem} 
\input xy
\xyoption{all}
\newtheorem{Theorem}{Theorem}[section]
\newtheorem{Lemma}[Theorem]{Lemma}
\newtheorem{Proposition}[Theorem]{Proposition}
\newtheorem{Corollary}[Theorem]{Corollary}
\newtheorem{Example}[Theorem]{Example}
\newtheorem{Remark}[Theorem]{Remark}
\newtheorem{Definition}[Theorem]{Definition}

\newtheorem{Notation}[Theorem]{Notation}
\newtheorem{Question}[Theorem]{Question}
\newtheorem*{Theorem A}{Theorem A}
\newcommand*{\longhookrightarrow}{\ensuremath{\lhook\joinrel\relbar\joinrel\rightarrow}}
\newcommand*{\overbar}[1]{\mkern 1.5mu\overline{\mkern-1.5mu#1\mkern-1.5mu}\mkern 1.5mu}
\advance\evensidemargin-.5in
\advance\oddsidemargin-.5in
\advance\textwidth1in

\begin{document}
\author{Charlie Beil}
 \address{Institut f\"ur Mathematik und Wissenschaftliches Rechnen, Universit\"at Graz, Heinrichstrasse 36, 8010 Graz, Austria.}
 \email{charles.beil@uni-graz.at}
 \title[The central nilradical of nonnoetherian dimer algebras]{The central nilradical of\\ nonnoetherian dimer algebras}
 \date{}

\begin{abstract}
Let $Z$ be the center of a nonnoetherian dimer algebra $A$ on a torus.
We show that the nilradical $\operatorname{nil}Z$ of $Z$ is prime, may be nonzero, and consists precisely of the central elements that vanish under a cyclic contraction of $A$.
This implies that the nonnoetherian scheme $\operatorname{Spec}Z$ is irreducible.
We also show that the reduced center $\hat{Z} = Z/\operatorname{nil}Z$ embeds into the center $R$ of the corresponding ghor algebra, and that their normalizations are equal.
Finally, we give three characterizations of the normality of $R$, and show that if $\hat{Z}$ is normal, then it has the special form $k + J$ where $J$ is an ideal of the cycle algebra of $A$.
\end{abstract}

\maketitle

\section{Introduction}

The main objective of this article is to establish certain key algebraic and geometric properties of the centers of nonnoetherian dimer algebras and ghor algebras on a torus.
A dimer algebra is a type of quiver algebra whose quiver embeds into a compact surface with homotopy-like relations defined on its paths.
Dimer algebras originated in string theory in 2005 \cite{HK, F-K, F-W, HV}, and since have found wide application to many areas of mathematics. 
Among these areas are noncommutative algebraic geometry \cite{B5,B7,Bo,Br,CQ,D,IN,IU,MR}, cluster algebras and categories \cite{BKM,GK,K,MS,P,RW}, mirror symmetry \cite{F-V,FU,HN}, and number theory \cite{BGH,H}.
The torus is special among other surfaces without boundary in that only on a torus do noetherian dimer algebras have exceptionally nice algebraic and homological properties as noncommutative crepant resolutions.
Ghor algebras, in contrast, remain suitably nice on higher genus surfaces \cite{BB1,BB2}.
Throughout, we will restrict our attention to dimer and ghor algebras on a torus.

A dimer algebra is noetherian if and only if its center is noetherian \cite[Theorem 1.1]{B4}.
Regardless of noetherianity, the Krull dimension of the center is three; if \textit{non}noetherian, then the center is the coordinate ring for a three-dimensional affine toric variety with a single `positive dimensional' point \cite[Theorem 1.1]{B6}. 
Recall that an integral domain is normal if it is integrally closed in its field of fractions.
A well known property of noetherian dimer algebras---being noncommutative crepant resolutions---is that their centers are normal integral domains.
In this article we consider the questions: \textit{Is the center of a nonnoetherian dimer algebra necessarily normal, or necessarily a domain?
If not, what can be said about its normalization, zero divisors, and nilradical?}

We will use two fundamental tools to answer these questions, both introduced in \cite{B2}: cyclic contractions and ghor algebras. 
A \textit{cyclic contraction} $\psi: A \to A'$ of a dimer algebra $A = kQ/I$ is the contraction of a particular set of arrows of $Q$ to vertices such that the cycles of $Q$ are preserved and the resulting dimer algebra $A'$ is noetherian (see (\ref{cycle algebra})). 
Remarkably, every (nondegenerate) dimer algebra admits a cyclic contraction \cite[Theorem 1.1]{B1}.
The \textit{ghor algebra} of $Q$ is the quotient
\begin{equation} \label{homotopy algebra}
\Lambda := A/\langle p - q \ | \ \text{$p,q$ a non-cancellative pair} \rangle,
\end{equation}
where a pair of distinct paths $p,q$ is said to be \textit{non-cancellative} if there is a path $r$ such that $rp = rq \not = 0$ or $pr = qr \not = 0$. 
A dimer algebra equals its ghor algebra if and only if it is noetherian \cite[Theorem 1.1]{B4}.
The center $R = Z(\Lambda)$ of $\Lambda$ will play an essential role in deciphering the structure of the center $Z$ of $A$. 
Our main theorem is the following.

\begin{Theorem} (\ref{nil}, \ref{subalgebra}, \ref{domains}, \ref{integral closure theorem}, \ref{nonnoetherian theorem}, \ref{R normal'}, \ref{normal corollary}.)
Let $A = kQ/I$ be a nonnoetherian dimer algebra on a torus with center $Z$, let $R$ be the center of its ghor algebra $\Lambda$, and let $\psi: A \to A'$ be any cyclic contraction of $A$. 
\begin{enumerate}
\item The nilpotent central elements of $A$ are precisely the central elements in the kernel of $\psi$,
\begin{equation*}
\operatorname{nil}Z = Z \cap \operatorname{ker} \psi.
\end{equation*}
\item The reduced center $\hat{Z} := Z/\operatorname{nil}Z$ is an integral domain.
The scheme $\operatorname{Spec}Z$ is therefore irreducible.
\item $\hat{Z}$ is a subalgebra of $R$, and their normalizations are equal and nonnoetherian.
\item $R$ is normal if and only if $R = k + J$ for some ideal $J$ of the center $Z'$ of $A'$.
Consequently, if $\hat{Z}$ is normal, then $\hat{Z} = k + J$.
\end{enumerate}
\end{Theorem}

We give examples of dimer algebras exhibiting various properties of the central nilradical. 
Notably, we show that $\operatorname{nil}Z$ may be nonzero (Example \ref{first example}); the containment $\hat{Z} \subseteq R$ may be proper (Example \ref{iso R}); and $\hat{Z}$ may not be normal (Proposition \ref{R is not normal}).

\section{Preliminaries} \label{definitions}

Throughout, $k$ is an algebraically closed field of characteristic zero.
Given a quiver $Q$, we denote by $kQ$ the path algebra of $Q$, and by $Q_{\ell}$ the paths of length $\ell$.
The vertex idempotent at vertex $i \in Q_0$ is denoted $e_i$, and the head and tail maps are denoted $\operatorname{h},\operatorname{t}: Q_1 \to Q_0$.
By monomial, we mean a non-constant monomial.

\newpage
\begin{Definition} \rm{ \
\begin{itemize}
 \item A \textit{dimer quiver} $Q$ is a quiver whose underlying graph $\overbar{Q}$ embeds into a compact surface $\Sigma$ such that each connected component of $\Sigma \setminus \overbar{Q}$ is simply connected and bounded by an oriented cycle of length at least $2$, called a \textit{unit cycle}.\footnote{The dual graph of a dimer quiver is called a dimer model or brane tiling, or, if on a disc, a plabic (= planar bicolored) graph \cite{P}.}
The \textit{dimer algebra} $A$ of $Q$ is the quotient $kQ/I$, where $I$ is the ideal
\begin{equation*} \label{I}
I := \left\langle p - q \ | \ \exists a \in Q_1 \ \text{s.t.\ $pa$ and $qa$ are unit cycles} \right\rangle \subset kQ,
\end{equation*}
and $p,q$ are paths.
Throughout, we will take $\Sigma$ to be a real two-torus.
 \item A \textit{perfect matching} $x$ of $Q$ is a set of arrows such that each unit cycle contains precisely one arrow in $x$.
A perfect matching $x$ is \textit{simple} if there is an oriented path between any two vertices in $Q \setminus x$.
In particular, $x$ is a simple matching if $Q \setminus x$ supports a simple $A$-module of dimension $1^{Q_0}$.
 \item A dimer quiver is \textit{nondegenerate} if each arrow is contained in a perfect matching.
Throughout, we will take all dimer quivers to be nondegenerate.
 \item If $p$ is a path in $Q$, then we refer to $p + I$ as a \textit{path} in $A$ since each representative of $p+I$ is a path.
 If $p,q$ are paths in $Q$ (resp.\ $A$) that are equal in $A$, then we will write $p \equiv q$ (resp.\ $p = q$).
 \item $A$ and $Q$ are \textit{non-cancellative} if there are paths $p,q,r \in A$ for which $p \not = q$, and
\begin{equation*} 
pr = qr \not = 0 \ \ \ \text{ or } \ \ \ rp = rq \not = 0;
\end{equation*}
 in this case, $p,q$ is called a \textit{non-cancellative pair}.
 Otherwise, $A$ and $Q$ are \textit{cancellative}; cancellativity was introduced in \cite{D}.
A (nondegenerate) dimer algebra is cancellative if and only if it is noetherian \cite[Theorem 1.1]{B4}. 
\end{itemize}
}\end{Definition}

\begin{Notation} \rm{
Let $\pi: \mathbb{R}^2 \rightarrow T^2$ be a covering map such that for some $i \in Q_0$,
\begin{equation*}
\pi(\mathbb{Z}^2) = i.
\end{equation*}
Denote by $Q^+ := \pi^{-1}(Q) \subset \mathbb{R}^2$ the covering quiver of $Q$.
For each path $p$ in $Q$, denote by $p^+$ the unique path in $Q^+$ with tail in the unit square $[0,1) \times [0,1) \subset \mathbb{R}^2$ satisfying $\pi(p^+) = p$.

For paths $p$, $q$ satisfying
\begin{equation} \label{exception}
\operatorname{t}(p^+) = \operatorname{t}(q^+) \ \ \ \text{ and } \ \ \ \operatorname{h}(p^+) = \operatorname{h}(q^+),
\end{equation}
denote by $\mathcal{R}_{p,q}$ the compact region in $\mathbb{R}^2 \supset Q^+$ bounded by (representatives of) $p^+$ and $q^+$, and denote by $\mathcal{R}_{p,q}^{\circ}$ the interior of $\mathcal{R}_{p,q}$.
}\end{Notation}

\begin{Notation} \rm{
By a \textit{cyclic subpath} of a path $p$, we mean a subpath of $p$ that is a nontrivial cycle.
Consider the following sets of cycles in $A$:
\begin{itemize}
 \item Let $\mathcal{C}$ be the set of cycles in $A$ (i.e., cycles in $Q$ modulo $I$).
 \item For $u \in \mathbb{Z}^2$, let $\mathcal{C}^u$ be the set of cycles $p \in \mathcal{C}$ such that
\begin{equation*}
\operatorname{h}(p^+) = \operatorname{t}(p^+) + u \in Q_0^+.
\end{equation*}
 \item For $i \in Q_0$, let $\mathcal{C}_i$ be the set of cycles in the vertex corner ring $e_iAe_i$.
 \item Let $\hat{\mathcal{C}}$ be the set of cycles $p \in \mathcal{C}$ such that $(p^2)^+$ does not have a cyclic subpath; or equivalently, the lift of each cyclic permutation of $p$ does not have a cyclic subpath.
\end{itemize}
We denote the intersection $\hat{\mathcal{C}} \cap \mathcal{C}^u \cap \mathcal{C}_i$, for example, by $\hat{\mathcal{C}}^u_i$.
Note that $\mathcal{C}^0$ is the set of cycles whose lifts are cycles in $Q^+$.
In particular, $\hat{\mathcal{C}}^0 = Q_0$.
Furthermore, the lift of any cycle $p$ in $\hat{\mathcal{C}}$ has no cyclic subpaths, although $p$ itself may have cyclic subpaths.
}\end{Notation}

\begin{Lemma} \label{4.13.2} \cite[Lemma 4.13.2]{B2}
Suppose paths $p^+$, $q^+$ have no cyclic subpaths modulo $I$, satisfy (\ref{exception}), and bound a region $\mathcal{R}_{p,q}$ that contains no vertices in its interior.
Then $p \equiv q$.
\end{Lemma}

Let $A = kQ/I$ be a dimer algebra.
For each perfect matching $x$ of $A$, consider the map
\begin{equation*}
n_x: Q_{\geq 0} \to \mathbb{Z}_{\geq 0}
\end{equation*}
defined by sending a path $p$ to the number of arrow subpaths of $p$ that are contained in $x$.
Observe that $n_x$ is additive on concatenated paths.
Furthermore, if $p,p' \in Q_{\geq 0}$ are paths satisfying $p + I = p' + I$, then $n_x(p) = n_x(p')$, by \cite[Lemma 2.1]{B2}.
In particular, $n_x$ induces a well-defined map on the paths of $A$.

Now consider dimer algebras $A = kQ/I$ and $A' = kQ'/I'$, and suppose $Q'$ is obtained from $Q$ by contracting a set of arrows $Q_1^* \subset Q_1$ to vertices.
This contraction defines a $k$-linear map of path algebras
\begin{equation*}
\psi: kQ \to kQ'.
\end{equation*}
If $\psi(I) \subseteq I'$, then $\psi$ induces a $k$-linear map of dimer algebras $\psi: A \to A'$, called a \textit{contraction}.\footnote{If, for example, $\psi$ contracts a unit cycle to a vertex, then $\psi(I) \not \subseteq I'$ by \cite[Lemma 3.5]{B2}.} 

To specify the structure we wish $\psi$ to preserve, consider the polynomial ring $k[\mathcal{S}']$ generated by the simple matchings $\mathcal{S}'$ of $A'$.
To each path $p \in A'$, associate the monomial
\begin{equation*} \label{tau bar def}
\bar{\tau}(p) := \prod_{x \in \mathcal{S}'} x^{n_x(p)} \in k[\mathcal{S}'].
\end{equation*}
For each $i,j \in Q'_0$, this association may be extended to a $k$-linear map $\bar{\tau}: e_jA'e_i \to k[\mathcal{S}']$, and is an algebra homomorphism if $i = j$.
Given $p \in e_jAe_i$ and $q \in e_{\ell}A'e_k$, we shall write
\begin{equation*}
\overbar{p} := \bar{\tau}_{\psi}(p) := \bar{\tau}(\psi(p)) \ \ \ \text{ and } \ \ \ \overbar{q} := \bar{\tau}(q).
\end{equation*}

$\psi$ is called a \textit{cyclic contraction} if $A'$ is cancellative and
\begin{equation} \label{cycle algebra}
S := k \left[ \cup_{i \in Q_0} \bar{\tau}_{\psi}(e_iAe_i) \right] = k \left[ \cup_{i \in Q'_0} \bar{\tau}(e_iA'e_i) \right] =: S'.
\end{equation}
The algebra $S$, called the \textit{cycle algebra}, is independent of the choice of cyclic contraction $\psi$ \cite[Theorem 3.14]{B3}.
$S$ is also isomorphic to the center of $A'$, and is a depiction of both the reduced center of $A$ and the center of $\Lambda$ \cite[Theorem 1.1]{B6}.
Surprisingly, every nondegenerate dimer algebra admits a cyclic contraction \cite[Theorem 1.1]{B1}.
Cyclic contractions and the cycle algebra were introduced in \cite{B2}.

In addition to the cycle algebra $S$, we will also consider the \textit{ghor center} of $A$,
\begin{equation*}
R := k\left[ \cap_{i \in Q_0} \bar{\tau}_{\psi}(e_iAe_i) \right] = \bigcap_{i \in Q_0} \bar{\tau}_{\psi}(e_iAe_i).
\end{equation*}
$R$ is isomorphic to the center of the ghor algebra $\Lambda$, given in (\ref{homotopy algebra}) \cite[Theorem 1.1]{B2}.

For $g, h$ in $R$ or $S$, we shall write $g \mid h$ if $g$ divides $h$ in the polynomial ring $k[\mathcal{S}']$.

The following lemmas will be useful.

\begin{Lemma} \label{from B}
Let $\psi: A \to A'$ be a cyclic contraction.
\begin{enumerate}
 \item If $p$ and $q$ are paths in $A$ (or $A'$) satisfying $qp \not = 0$, then $\overbar{qp} = \overbar{q} \overbar{p}$.
 \item For each $i,j \in Q'_0$, the $k$-linear map $\bar{\tau}: e_jA'e_i \to k[\mathcal{S}']$ is injective.
\end{enumerate}
\end{Lemma}

\begin{proof}
(1) holds since for each simple matching $x \in \mathcal{S}$, the map $n_x$ is additive on concatenated paths.
(2) holds by \cite[Proposition 4.30]{B2}.
\end{proof}

\begin{Lemma} \label{unit cycle lemma}
If $\sigma_i$, $\sigma'_i$ are unit cycles at $i \in Q_0$, then $\sigma_i = \sigma'_i$ in $A$.
Furthermore, the element $\sum_{i \in Q_0} \sigma_i \in A$ is central.
\end{Lemma}

\begin{proof}
Clear.
\end{proof}

We denote by $\sigma_i$ the (unique) unit cycle at $i \in Q_0$ modulo $I$, and by $\sigma$ the monomial
\begin{equation*}
\sigma := \overbar{\sigma}_i = \prod_{x \in \mathcal{S}'}x.
\end{equation*}

\begin{Lemma} \label{p sigma = q sigma} \cite[Lemma 4.3.1]{B2}
If $p,q \in e_jAe_i$ are paths satisfying (\ref{exception}), then there is an $m,n \geq 0$ such that $p \sigma_i^m = q \sigma_i^n$.
\end{Lemma}

\begin{Lemma} \label{cyclelemma}
Suppose $Q$ admits a cyclic contraction. 
Let $u,v \in \mathbb{Z}^2$ and let $p \in \mathcal{C}^u$, $q \in \mathcal{C}^v$ be cycles.
\begin{enumerate}
 \item $u = 0$ if and only if $\overbar{p} = \sigma^m$ for some $m \geq 1$.
 \item $u = v$ if and only if $\overbar{p} = \overbar{q}\sigma^m$ for some $m \in \mathbb{Z}$.
 \item $p \not \in \hat{\mathcal{C}}$ if and only if $\sigma \mid \overbar{p}$.
\end{enumerate}
\end{Lemma}

\begin{proof}
({$\text{1}, \Rightarrow$}): \cite[Lemma 5.2]{B2}.

({$\text{2}, \Rightarrow$}): \cite[Lemma 4.19]{B2}.

({$\text{2}, \Leftarrow$}):
Suppose $\overbar{p} = \overbar{q} \sigma^m$ for some $m \in \mathbb{Z}$.
Let $r$ (resp.\ $s$) be a path whose lift $r^+$ ($s^+$) has tail (head) $\operatorname{h}(p^+)$ and head (tail) $\operatorname{h}(q^+)$.
Since $\overbar{p} = \overbar{q} \sigma^m$, there is some $\ell \in \mathbb{Z}$ such that $\overbar{r} = \sigma^{\ell}$.
In particular, $\bar{\tau}(\psi(r)) = \sigma^{\ell}$.
Thus, $\psi(r)$ is in $\mathcal{C}'^0$ since $A'$ is cancellative \cite[Lemma 4.29]{B2}.
Whence, $\psi$ contracts $s$ to a vertex.
Furthermore, $\psi$ does not contract any cycle to a vertex since $\psi$ is a cyclic contraction \cite[Lemma 3.6.1]{B2}.
Therefore $s$ is a vertex.
But then $r$ is in $\mathcal{C}^0$ and $\operatorname{h}(p^+) = \operatorname{h}(q^+)$.
Consequently, $u = v$.

({$\text{1}, \Leftarrow$}): Follows from ({$\text{2}, \Leftarrow$}) with $q = e_{\operatorname{t}(p)}$. 

({$\text{3}, \Rightarrow$}): \cite[Lemma 5.2]{B2}.

({$\text{3}, \Leftarrow$}):
First recall that if $\psi(p)$ is not in $\hat{\mathcal{C}}'$, then the lift $(\psi(p)^{2})^+$ has a cyclic subpath (by definition).
Whence, the lift $(p^{2})^+$ has a cyclic subpath (though the converse need not hold).
Thus, $p$ is not in $\hat{\mathcal{C}}$.
Equivalently, if $p$ is in $\hat{\mathcal{C}}$, then $\psi(p)$ is in $\hat{\mathcal{C}}'$.
But then $\sigma \nmid \bar{\tau}(\psi(p)) = \overbar{p}$ since $A'$ is cancellative, by \cite[Proposition 4.21.1]{B2}.
\end{proof}

\begin{figure}
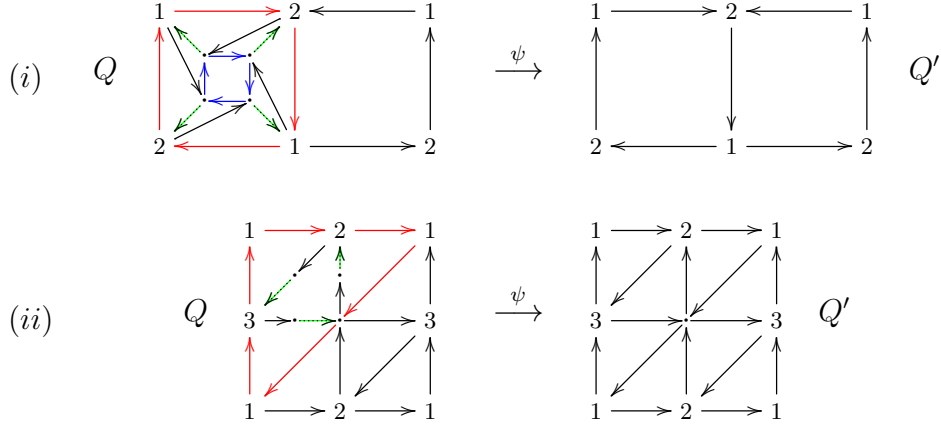

\begin{equation*}
\begin{array}{lc}
(i) &
\begin{array}{rcccl}
Q &
\xy
(-18,9)*+{\text{\scriptsize{$1$}}}="1";(0,9)*+{\text{\scriptsize{$2$}}}="2";(18,9)*+{\text{\scriptsize{$1$}}}="3";
(-18,-9)*+{\text{\scriptsize{$2$}}}="4";(0,-9)*+{\text{\scriptsize{$1$}}}="5";(18,-9)*+{\text{\scriptsize{$2$}}}="6";
(-12,3)*{\cdot}="7";(-6,3)*{\cdot}="8";(-12,-3)*{\cdot}="9";(-6,-3)*{\cdot}="10";
{\ar@{->}@[red]"1";"2"};{\ar@{->}@[red]"2";"5"};{\ar@{->}@[red]"5";"4"};{\ar@{->}@[red]"4";"1"};
{\ar@{->}"6";"3"};{\ar@{->}"3";"2"};{\ar@{->}"5";"6"};
{\ar@{->}@[green]"7";"1"};{\ar@{..>}"7";"1"};{\ar@{->}@[green]"8";"2"};{\ar@{..>}"8";"2"};
{\ar@{->}@[green]"9";"4"};{\ar@{..>}"9";"4"};{\ar@{->}@[green]"10";"5"};{\ar@{..>}"10";"5"};
{\ar@{->}@[blue]"7";"8"};{\ar@{->}@[blue]"8";"10"};{\ar@{->}@[blue]"10";"9"};{\ar@{->}@[blue]"9";"7"};
{\ar@{->}"1";"9"};{\ar@{->}"2";"7"};{\ar@{->}"5";"8"};{\ar@{->}"4";"10"};
\endxy
& \ \ \stackrel{\psi}{\longrightarrow} \ \ &
\xy
(-18,9)*+{\text{\scriptsize{$1$}}}="1";(0,9)*+{\text{\scriptsize{$2$}}}="2";(18,9)*+{\text{\scriptsize{$1$}}}="3";
(-18,-9)*+{\text{\scriptsize{$2$}}}="4";(0,-9)*+{\text{\scriptsize{$1$}}}="5";(18,-9)*+{\text{\scriptsize{$2$}}}="6";
{\ar@{->}"1";"2"};{\ar@{->}"2";"5"};{\ar@{->}"5";"4"};{\ar@{->}"4";"1"};
{\ar@{->}"6";"3"};{\ar@{->}"3";"2"};{\ar@{->}"5";"6"};
\endxy
& Q'
\end{array}
\\ \\
(ii) &
\begin{array}{rcccl}
Q &
\xy
(0,0)*{\cdot}="5";(-12,12)*+{\text{\scriptsize{$1$}}}="1";(0,12)*+{\text{\scriptsize{$2$}}}="2";(12,12)*+{\text{\scriptsize{$1$}}}="3";
(12,0)*+{\text{\scriptsize{$3$}}}="6";(12,-12)*+{\text{\scriptsize{$1$}}}="9";(0,-12)*+{\text{\scriptsize{$2$}}}="8";
(-12,-12)*+{\text{\scriptsize{$1$}}}="7";(-12,0)*+{\text{\scriptsize{$3$}}}="4";
(0,6)*{\cdot}="11";(-6,6)*{\cdot}="10";(-6,0)*{\cdot}="12";
{\ar@{->}@[red]"1";"2"};{\ar@{->}@[red]"2";"3"};{\ar@{->}@[red]"3";"5"};{\ar@{->}@[red]"5";"7"};{\ar@{->}@[red]"7";"4"};{\ar@{->}@[red]"4";"1"};
{\ar@{->}"6";"3"};{\ar@{->}"9";"6"};{\ar@{->}"8";"5"};{\ar@{->}"4";"12"};{\ar@{->}"5";"6"};{\ar@{->}"7";"8"};
{\ar@{->}"8";"9"};{\ar@{->}"5";"11"};{\ar@{->}"2";"10"};
{\ar@{->}@[green]"10";"4"};{\ar@{..>}"10";"4"};{\ar@{->}@[green]"11";"2"};{\ar@{..>}"11";"2"};{\ar@{->}@[green]"12";"5"};{\ar@{..>}"12";"5"};{\ar"6";"8"};
\endxy
& \ \ \stackrel{\psi}{\longrightarrow} \ \ &
\xy
(0,0)*{\cdot}="5";(-12,12)*+{\text{\scriptsize{$1$}}}="1";(0,12)*+{\text{\scriptsize{$2$}}}="2";(12,12)*+{\text{\scriptsize{$1$}}}="3";
(12,0)*+{\text{\scriptsize{$3$}}}="6";(12,-12)*+{\text{\scriptsize{$1$}}}="9";(0,-12)*+{\text{\scriptsize{$2$}}}="8";
(-12,-12)*+{\text{\scriptsize{$1$}}}="7";(-12,0)*+{\text{\scriptsize{$3$}}}="4";
{\ar@{->}"1";"2"};{\ar@{->}"2";"3"};{\ar@{->}"3";"5"};{\ar@{->}"5";"7"};{\ar@{->}"7";"4"};{\ar@{->}"4";"1"};
{\ar@{->}"6";"3"};{\ar@{->}"9";"6"};{\ar@{->}"8";"5"};{\ar@{->}"4";"5"};{\ar@{->}"5";"6"};{\ar@{->}"7";"8"};
{\ar@{->}"8";"9"};{\ar@{->}"5";"2"};{\ar@{->}"2";"4"};{\ar"6";"8"};
\endxy
& Q'
\end{array}
\end{array}
\end{equation*}
\caption{Examples for Remarks \ref{first remark} and Proposition \ref{R is not normal}.
The quivers are drawn on a torus, the contracted arrows are drawn in green, and the 2-cycles have been removed from $Q'$.
In each example, the cycle in $Q$ formed from the red arrows is not equal to a product of unit cycles (modulo $I$).
However, in example (i) this cycle is mapped to a unit cycle in $Q'$ under $\psi$.}
\label{holy smokes batman}
\end{figure}

\begin{Remark} \label{first remark} \rm{
Let $p^+$ be a cycle in $Q^+$; then $\overbar{p} = \sigma^m$ for some $m \geq 0$ by Lemma \ref{cyclelemma}.1.
However, $p$ may not necessarily equal a power of the unit cycle $\sigma_{\operatorname{t}(p)}$ (modulo $I$).
Two examples are given by the red cycles in Figures \ref{holy smokes batman}.i and \ref{holy smokes batman}.ii.

Furthermore, it is possible for two cycles in $Q^+$, distinct modulo $I$, and one of which is properly contained in the region bounded by the other, to have equal $\bar{\tau}_{\psi}$-images.
Indeed, consider Figure \ref{holy smokes batman}.i: the red cycle and the unit cycle in its interior both have $\bar{\tau}_{\psi}$-image $\sigma$.
}\end{Remark}

\section{The central nilradical from cyclic contractions}

Let $A = kQ/I$ be a dimer algebra with center $Z$.
In this section we will show that the nilpotent elements in $Z$ are precisely the central elements that vanish under a cyclic contraction.
We regard this as \textit{the} fundamental result in the structure of nonnoetherian dimer algebras on a torus. 

Noetherian dimer algebras are prime \cite[Theorem 4.31, Corollary 5.12]{B2}, and therefore their centers are reduced.
In the following two examples we show that nonnoetherian dimer algebras may have non-reduced centers, and consequently that dimer algebras need not be prime.

\begin{Example} \label{first example} \rm{
Consider the nonnoetherian dimer algebra $A$ with quiver $Q$ given in Figure \ref{non-cancellative central}.
(A cyclic contraction of $A$ is given in Figure \ref{deformation figure0}.)
The paths $p$, $q$, $a$ satisfy
\begin{equation*}
z := (p-q)a + a(p-q) \in \operatorname{nil}Z.
\end{equation*}
In particular, $\operatorname{nil}Z \not = 0$.
$A$ is therefore not prime since
\begin{equation*}
zAz = z^2A = 0.
\end{equation*}
We note that $A$ also contains non-central elements $a,b$ with the property that $aAb = 0$; for example, $(p-q)Ae_1 = 0$.
}\end{Example}

\begin{Example} \label{second example} \rm{
Let $Q$ be a dimer quiver containing the subquiver given in Figure \ref{z22=0}.
Given any cyclic contraction $\psi: A \to A'$, the $\psi$-image of the cycle $st$ is a unit cycle in $Q'$.
Set $p := cbtba$.
Then
\begin{equation*}
p + \sum_{j \in Q_0 \setminus \{i\}} \sigma_j^2 \ \ \ \text{ and } \ \ \ z := p - \sigma_i^2
\end{equation*}
are nonzero central elements of $A$, by Lemma \ref{unit cycle lemma}.
Furthermore, $z^2 = 0$, and so $z$ is in the central nilradical of $A$.
}\end{Example}

\begin{figure}
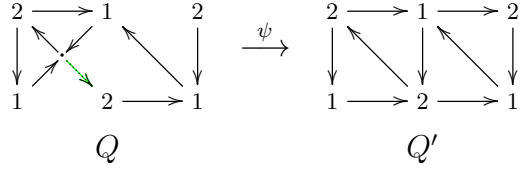

\begin{equation*}
\begin{array}{ccc}
\xy
(-12,6)*+{\text{\scriptsize{$2$}}}="1";(0,6)*+{\text{\scriptsize{$1$}}}="2";(12,6)*+{\text{\scriptsize{$2$}}}="3";
(-12,-6)*+{\text{\scriptsize{$1$}}}="4";(0,-6)*+{\text{\scriptsize{$2$}}}="5";(12,-6)*+{\text{\scriptsize{$1$}}}="6";
(-6,0)*{\cdot}="7";
{\ar^{}"1";"4"};{\ar^{}"4";"7"};{\ar^{}"7";"1"};{\ar^{}"1";"2"};{\ar^{}"2";"7"};{\ar^{}"5";"6"};{\ar^{}"6";"2"};{\ar^{}"3";"6"};
{\ar@[green]"7";"5"};{\ar@{..>}"7";"5"};\endxy
 & \stackrel{\psi}{\longrightarrow} &
\xy
(-12,6)*+{\text{\scriptsize{$2$}}}="1";(0,6)*+{\text{\scriptsize{$1$}}}="2";(12,6)*+{\text{\scriptsize{$2$}}}="3";
(-12,-6)*+{\text{\scriptsize{$1$}}}="4";(0,-6)*+{\text{\scriptsize{$2$}}}="5";(12,-6)*+{\text{\scriptsize{$1$}}}="6";
{\ar^{}"1";"4"};{\ar^{}"4";"5"};{\ar^{}"5";"1"};{\ar^{}"1";"2"};{\ar^{}"2";"5"};{\ar^{}"5";"6"};{\ar^{}"6";"2"};{\ar^{}"2";"3"};{\ar^{}"3";"6"};
\endxy
\\
Q & \ \ \ \ & Q' \\
\end{array}
\end{equation*}
\caption{The nonnoetherian dimer algebra $A = kQ/I$ cyclically contracts to the noetherian dimer algebra $A' = kQ'/I'$.
Both quivers are drawn on a torus and the contracted arrow is drawn in green.
Here, the cycle algebra of $A$ is $S = k\left[ x^2, y^2, xy, z \right] \subset k[\mathcal{S}'] = k[x,y,z]$, and the ghor center of $A$ is $R = k + (x^2, y^2, xy)S$.}
\label{deformation figure0}
\end{figure}

\begin{figure}
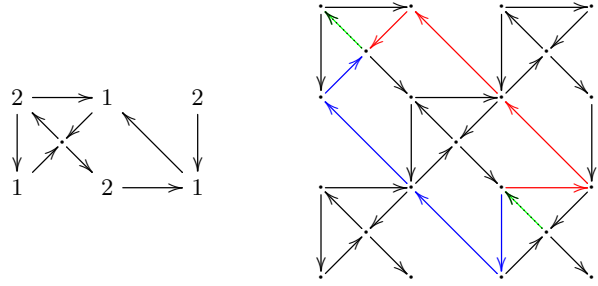

\begin{equation*}
\xy
(-12,6)*+{\text{\scriptsize{$2$}}}="1";(0,6)*+{\text{\scriptsize{$1$}}}="2";(12,6)*+{\text{\scriptsize{$2$}}}="3";
(-12,-6)*+{\text{\scriptsize{$1$}}}="4";(0,-6)*+{\text{\scriptsize{$2$}}}="5";(12,-6)*+{\text{\scriptsize{$1$}}}="6";
(-6,0)*{\cdot}="7";
{\ar^{}"1";"4"};{\ar^{}"4";"7"};{\ar^{}"7";"1"};{\ar^{}"1";"2"};{\ar^{}"2";"7"};{\ar^{}"5";"6"};{\ar^{}"6";"2"};{\ar^{}"3";"6"};
{\ar"7";"5"};
\endxy \ \ \ \ \ \ \ \ \ \
\xy
(-18,18)*{\cdot}="1";(-6,18)*{\cdot}="2";(6,18)*{\cdot}="3";(18,18)*{\cdot}="4";
(-12,12)*{\cdot}="5";(12,12)*{\cdot}="6";
(-18,6)*{\cdot}="7";(-6,6)*{\cdot}="8";(6,6)*{\cdot}="9";(18,6)*{\cdot}="10";
(0,0)*{\cdot}="11";
(-18,-6)*{\cdot}="12";(-6,-6)*{\cdot}="13";(6,-6)*{\cdot}="14";(18,-6)*{\cdot}="15";
(-12,-12)*{\cdot}="16";(12,-12)*{\cdot}="17";
(-18,-18)*{\cdot}="18";(-6,-18)*{\cdot}="19";(6,-18)*{\cdot}="20";(18,-18)*{\cdot}="21";
{\ar"1";"2"};{\ar"3";"4"};{\ar"1";"7"};{\ar@[teal]|-a"5";"1"};{\ar"5";"8"};{\ar"3";"9"};{\ar"6";"3"};{\ar"6";"10"};{\ar"4";"6"};{\ar"9";"6"};
{\ar"8";"9"};{\ar"8";"13"};{\ar"13";"11"};{\ar"11";"8"};{\ar"11";"14"};{\ar"10";"15"};
{\ar"12";"13"};{\ar"12";"18"};{\ar"18";"16"};{\ar"16";"12"};{\ar"16";"19"};{\ar"13";"16"};{\ar"20";"17"};{\ar"15";"17"};{\ar"17";"21"};
{\ar@[red]"14";"15"};{\ar@[red]"15";"9"};{\ar@[red]"9";"2"};{\ar@[red]"2";"5"};{\ar@[teal]|-a"17";"14"};{\ar@[blue]"14";"20"};{\ar@[blue]"20";"13"};{\ar@[blue]"13";"7"};{\ar@[blue]"7";"5"};
{\ar"9";"11"};
\endxy
\end{equation*}
\caption{The dimer algebra $A$ given in Figure \ref{deformation figure0} has a non-vanishing central nilradical, $\operatorname{nil}Z \not = 0$.
A fundamental domain of $Q$ is shown on the left and a larger region of $Q^+$ is shown on the right.
The paths $p$, $q$, $a$ are drawn in red, blue, and teal respectively.
The element $(p-q)a + a(p-q)$ is central and squares to zero.}
\label{non-cancellative central}
\end{figure}

\begin{figure}
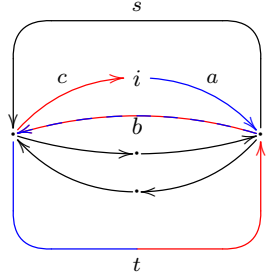

\begin{equation*}
\xy 0;/r.6pc/:
(-6.5,0)*{\cdot}="1";(-6.5,4)*{}="2";
(-4.5,6)*{}="3";(4.5,6)*{}="4";
(6.5,4)*{}="5";(6.5,0)*{\cdot}="6";(6.5,-4)*{}="7";
(4.5,-6)*{}="8";(-4.5,-6)*{}="9";
(-6.5,-4)*{}="10";
(0,3)*+{\text{\scriptsize{$i$}}}="11";(0,1)*{}="12";(0,-1)*{\cdot}="13";(0,-3)*{\cdot}="14";
(0,-6)*{}="15";
(0,0.5)*{\text{{\scriptsize $b$}}}="";
{\ar@{<-}"1";"2"};{\ar@{-}@/^.45pc/"2";"3"};{\ar@{-}^s"3";"4"};{\ar@{-}@/^.45pc/"4";"5"};{\ar@{-}"5";"6"};
{\ar@{<-}"6";"7"};{\ar@{-}@/^.45pc/"7";"8"};{\ar@{-}"8";"15"};{\ar@{-}"9";"15"};{\ar@{-}@/^.45pc/"9";"10"};{\ar@{-}"10";"1"};
{\ar@{}^t"8";"9"};
{\ar@/^/^c"1";"11"};{\ar@/^/^a"11";"6"};
{\ar@/^/"6";"14"};{\ar@/^/"14";"1"};
{\ar@{-}@/_.15pc/"6";"12"};{\ar@/_.15pc/"12";"1"};
{\ar@/_.15pc/"13";"6"};{\ar@/_.15pc/"1";"13"};
\endxy
\end{equation*}
\caption{The subquiver of $Q$ in Example \ref{second example}.
The paths $a$ and $c$ are arrows in $Q$, and all other paths are paths of positive length.
Setting $p := cbtba$, the elements $p + \sum_{j \in Q_0 \setminus \{i\}} \sigma^2_j$ and $z := p - \sigma^2_i$ are in the center of $A$.
Furthermore, $z^2 = 0$, and so $A$ has a nonvanishing central nilradical.}
\label{z22=0}
\end{figure}

\begin{figure}
\begin{equation*}
\begin{array}{rl}
\text{\footnotesize{$(a.i)$}} & \xy 0;/r.4pc/:
(-30,0)*{\cdot}="2";(10,0)*{\cdot}="4";
(-25,-5)*{}="6";(-5,-5)*{\cdot}="7";(15,-5)*{}="8";(35,-5)*{\cdot}="9";
(-25,5)*{\cdot}="10";(-5,5)*{}="11";(15,5)*{\cdot}="12";(35,5)*{}="13";
(-40,0)*+{\text{\scriptsize{$i$}}}="1";(0,0)*+{\text{\scriptsize{$i$}}}="3";(40,0)*+{\text{\scriptsize{$i$}}}="5";
{\ar^r"1";"2"};{\ar^r"3";"4"};
{\ar@[red]@{-}"2";"6"};{\ar@[red]@/_1pc/^{p''}"6";"7"};{\ar@[red]_a"7";"3"};
{\ar@[red]@{-}"4";"8"};{\ar@[red]@/_1pc/^{p''}"8";"9"};{\ar@[red]_a"9";"5"};
{\ar@[blue]^{b}"2";"10"};{\ar@{-}@[blue]@/^1pc/_{q''}"10";"11"};{\ar@[blue]"11";"3"};
{\ar@[blue]^{b}"4";"12"};{\ar@{-}@[blue]@/^1pc/_{q''}"12";"13"};{\ar@[blue]"13";"5"};
\endxy
\\
& \\
\text{\footnotesize{$(a.ii)$}} &
\xy 0;/r.4pc/:
(-35,0)*{\cdot}="17";(-20,0)*{\cdot}="2";(-5,0)*{\cdot}="3";(5,0)*{\cdot}="5";(20,0)*{\cdot}="6";(35,0)*{\cdot}="7";
(-40,5)*{\cdot}="9";(0,5)*{\cdot}="11";(0,-5)*{\cdot}="14";(40,-5)*{\cdot}="16";
(-20,5)*{}="10";(20,5)*{}="12";(-20,-5)*{}="13";(20,-5)*{}="15";
(-40,0)*+{\text{\scriptsize{$i$}}}="1";(0,0)*+{\text{\scriptsize{$i$}}}="4";(40,0)*+{\text{\scriptsize{$i$}}}="8";
{\ar@[blue]_{\beta'}"1";"17"};{\ar@[blue]_{q'}"17";"2"};{\ar@[blue]"2";"3"};{\ar@[blue]_{\beta}"3";"4"};
{\ar@[blue]^{\beta'}"4";"5"};{\ar@[blue]_{q'}"5";"6"};{\ar@[blue]"6";"7"};{\ar@[blue]_{\beta}"7";"8"};
{\ar@[red]_{\alpha}"1";"9"};
{\ar@/^1pc/@{-}@[red]_{p'}"9";"10"};{\ar@[red]"10";"2"};
{\ar@{-}@[red]"2";"13"};
{\ar@/_1pc/@[red]"13";"14"};{\ar@[red]^{\alpha'}"14";"4"};
{\ar@[red]_{\alpha}"4";"11"};
{\ar@{-}@/^1pc/@[red]_{p'}"11";"12"};{\ar@[red]"12";"6"};
{\ar@{-}@[red]"6";"15"};
{\ar@/_1pc/@[red]"15";"16"};{\ar@[red]^{\alpha'}"16";"8"};
{\ar@[teal]|-{c}"11";"3"};{\ar@[teal]|-{c'}"5";"14"};
{\ar@/^1.5pc/^{\gamma}"3";"11"};{\ar@/_1.5pc/_{\gamma'}"14";"5"};
\endxy
\\
& \\
\text{\footnotesize{$(b.i)$}} &
\xy 0;/r.4pc/:
(-40,5)*{\cdot}="1a";(-30,5)*{\cdot}="2a";(-35,-5)*{\cdot}="3a";
(-5,5)*{\cdot}="1b";(5,5)*{\cdot}="2b";(0,-5)*{\cdot}="3b";
(40,5)*{\cdot}="2c";(30,5)*{\cdot}="1c";(35,-5)*{\cdot}="3c";
{\ar@[blue]^{q'}"2a";"1b"};{\ar@[blue]^{q'}"2b";"1c"};
{\ar^{\beta = \beta'}"1a";"2a"};{\ar^{\beta = \beta'}"1b";"2b"};{\ar^{\beta = \beta'}"1c";"2c"};
{\ar@[red]_{\alpha}"2a";"3a"};{\ar@[red]|-{\alpha' = c}"3b";"1b"};{\ar@[red]|-{\alpha = c'}"2b";"3b"};{\ar@[red]_{\alpha'}"3c";"1c"};
{\ar@[red]^{p''}"3a";"3b"};{\ar@[red]^{p''}"3b";"3c"};
{\ar@/_2.1pc/_{\gamma'}"3b";"2b"};{\ar@/_2.1pc/_{\gamma}"1b";"3b"};
\endxy
\\
& \\
\text{\footnotesize{$(b.ii)$}} &
\xy 0;/r.4pc/:
(-35,7.5)*+{\text{\scriptsize{$i$}}}="1a";(-40,2.5)*{\cdot}="2a";(-40,-2.5)*{}="4a";(-35,-7.5)*{\cdot}="6a";
(-30,7.5)*{\cdot}="7a";(-5,7.5)*{\cdot}="8a";
(0,7.5)*+{\text{\scriptsize{$i$}}}="1b";(-5,2.5)*{\cdot}="2b";(5,2.5)*{\cdot}="3b";(-5,-2.5)*{}="4b";(5,-2.5)*{}="5b";(0,-7.5)*{\cdot}="6b";
(5,7.5)*{\cdot}="7b";(30,7.5)*{\cdot}="8b";
(35,7.5)*+{\text{\scriptsize{$i$}}}="1c";(40,2.5)*{\cdot}="3c";(40,-2.5)*{}="5c";(35,-7.5)*{\cdot}="6c";
{\ar@[blue]^{\beta'}"1a";"7a"};{\ar@[blue]^{q'}"7a";"8a"};{\ar@[blue]^{\beta}"8a";"1b"};
{\ar@[blue]^{\beta'}"1b";"7b"};{\ar@[blue]^{q'}"7b";"8b"};{\ar@[blue]^{\beta}"8b";"1c"};
{\ar@[red]_{\alpha}"1a";"2a"};{\ar@[red]@{-}"2a";"4a"};{\ar@[red]"4a";"6a"};
{\ar@[red]^{p'}"6a";"6b"};
{\ar@[red]|-{\alpha}"1b";"2b"};{\ar@[red]@{-}"2b";"4b"};{\ar@[red]"4b";"6b"};{\ar@[red]@{-}"6b";"5b"};{\ar@[red]"5b";"3b"};{\ar@[red]|-{\alpha'}"3b";"1b"};
{\ar@[red]^{p'}"6b";"6c"};
{\ar@[red]@{-}"6c";"5c"};{\ar@[red]"5c";"3c"};{\ar@[red]_{\alpha'}"3c";"1c"};
{\ar@[teal]|-{c}"2b";"8a"};{\ar@[teal]|-{c'}"3b";"7b"};
{\ar@/_1.7pc/_{\gamma}"8a";"2b"};{\ar@/^1.7pc/^{\gamma'}"7b";"3b"};
\endxy
\end{array}
\end{equation*}
\caption{The different possible cases in the proof of Proposition \ref{Lucy}.
In ($a.i$), $r$ may be vertex.}
\label{nguzka}
\end{figure}

\begin{Lemma} \label{not central}
Let $i \in Q_0$, and suppose $z \in A$ is a central element for which $ze_i = p - q$.
Then
\begin{equation*}
pq = qp.
\end{equation*}
\end{Lemma}

\begin{proof}
Since $z$ is central, we have
\begin{equation*}
p^2 - pq = p(p - q) = pz = zp = (p - q)p  = p^2 - qp.
\end{equation*}
Whence $pq = qp$.
\end{proof}

\begin{Proposition} \label{Lucy}
Let $z \in Z$ and $i \in Q_0$, and suppose there is a non-cancellative pair of cycles $p,q \in e_iAe_i$ such that $ze_i = p - q$.
Then
\begin{equation*} \label{ivo}
p^2 = pq = qp = q^2.
\end{equation*}
\end{Proposition}

\begin{proof}
In the following, by `path' or `cycle' we mean a path or cycle in $Q$ (not modulo $I$). 
If $a$ is an arrow and $s,t$ are paths such that $as,at$ are unit cycles, then $s$ is called an `arc' and $t$ its `complementary arc'.

Let $p, q \in kQ$ be representative paths of $p+I, q+I \in A$.
To prove the lemma, it suffices to show that $p^2 \equiv pq$, since $qp \equiv pq$ by Lemma \ref{not central}.

If $p = \sigma_i^n$ for some $n \geq 1$, then $p^2 \equiv pq$, by Lemma \ref{unit cycle lemma}.
(Such cases exist; see Example \ref{second example}.)
So suppose $p$ is not a power of a unit cycle.

Since $qp \equiv pq$ by Lemma \ref{not central}, we may assume that the representatives $p,q$ factor into paths 
\begin{equation} \label{p q factor}
p = \alpha' p'\alpha \ \ \ \text{ and } \ \ \ q = \beta q' \beta',
\end{equation}
where $\alpha, \alpha', \beta, \beta' \in Q_{\geq 1}$ are subpaths of unit cycles and $\alpha \beta$ and $\beta' \alpha'$ are arcs.
Let $\gamma, \gamma'$ be their complementary arcs: 
\begin{equation} \label{p q factor2}
\alpha \beta \equiv \gamma \ \ \ \text{ and } \ \ \ \beta' \alpha' \equiv \gamma'.
\end{equation}

There are two main cases to consider.

($a$) First assume that $p^{2+}$ does not intersect itself.

($a.i$) Consider the setup given in Figure \ref{nguzka}.$a.i$, where $p,q$ factor into paths 
\begin{equation*}
p = ap''r \ \ \ \text{ and }  \ \ \ q = q''br,
\end{equation*} 
with $a,b \in Q_1$ and $r \in Q_{\geq 0}$.

If $rp \equiv rq$, then 
\begin{equation*}
p^2 = ap''rp \equiv ap''rq = pq,
\end{equation*} 
which is what we wanted to show.

So suppose $p^2 \not \equiv pq$; then $rp \not \equiv rq$.
Thus, by our choice of representatives $p,q$ satisfying (\ref{p q factor}), $bra$ must be a subpath of a unit cycle. 
However, it is clear from the figure that this is not possible.

($a.ii$) Since case ($i$) is not possible and $p,q$ factor into the paths (\ref{p q factor}), we have the setup shown in Figure \ref{nguzka}.$a.ii$. 
From (\ref{p q factor2}) we have
\begin{equation*}
qp \equiv \beta q' \gamma' p' \alpha \ \ \ \text{ and } \ \ \ pq \equiv \alpha' p' \gamma q' \beta'.
\end{equation*}
But $(\beta q' \gamma' p' \alpha)^+$ and $(\alpha' p' \gamma q' \beta')^+$ have moved \textit{away} from $(pq)^+$ and $(qp)^+$ respectively.
Similarly, further applications of the dimer relations $I$ only homotope $(pq)^+$ and $(qp)^+$ further away from each other.
Consequently, it is not possible that $qp \equiv pq$ in this case, a contradiction.

($b$) Now assume that $p^{2+}$ intersects itself.

($b.i$) First suppose $p,q$ share a common leftmost (or rightmost) nontrivial subpath $\beta \in Q_{\geq 1}$.
Then there are paths $p', q' \in Q_{\geq 1}$ such that 
\begin{equation*}
p_1 := p = \beta p' \ \ \ \ \text{ and } \ \ \ \ q_1 := q = \beta q'.
\end{equation*}
Set
\begin{equation*}
p_2 := p' \beta \ \ \ \ \text{ and } \ \ \ \ q_2 := q' \beta.
\end{equation*}

Let $z \in Z$ be such that $z e_{\operatorname{h}(\beta)} = p_1 - q_1 + I$.
Then, since $z \beta = \beta z$, we have $ze_{\operatorname{t}(\beta)} = p_2 - q_2 + I$.
Therefore, by Lemma \ref{not central},
\begin{equation*}
p_1q_1 \equiv q_1p_1 \ \ \ \ \text{ and } \ \ \ \ p_2q_2 \equiv q_2p_2.
\end{equation*}
It thus suffices to assume that $p'$ factors into paths $p' = \alpha' p'' \alpha$, that is, 
\begin{equation*}
p_1 = \beta \alpha' p'' \alpha \ \ \ \ \text{ and } \ \ \ \ p_2 = \alpha' p'' \alpha \beta,
\end{equation*}
where $\alpha \beta$ is an arc subpath of $p_1q_1$ and $\beta \alpha'$ is an arc subpath of $q_2p_2$.

Since $\alpha \beta$ and $\beta \alpha'$ are both arcs, $\alpha \beta \alpha'$ must be a unit cycle with $\alpha, \alpha'$ arrows.
We therefore have the setup shown in Figure \ref{nguzka}.b.i.
Here, $\gamma, \alpha \beta$ are complementary arcs, and $\gamma', \beta \alpha'$ are complementary arcs. 

In order to homotope the path $p_1q_1$ to $q_1p_1$, we first use the relation $\alpha \beta \equiv \gamma$.
Continuing, we obtain
\begin{equation*}
p_1q_1 = (\beta \alpha' p'' \alpha) (\beta q') \equiv \beta \alpha' p'' \gamma q' \equiv \beta \alpha' p'' p'' \alpha \sigma^{\ell}_{\operatorname{h}(\beta)}
\end{equation*}
for some $\ell \geq 0$, by Lemma --.
But $\overbar{p}_1 = \overbar{q}_1$ since $p_1,q_1$ is a non-cancellative pair.
Whence, $\ell = 1$.
Therefore,
\begin{equation*}
pq = p_1q_1 \equiv \beta \alpha' p'' p'' \alpha \sigma_{\operatorname{h}(\beta)} \equiv \beta \alpha' p'' \sigma_{\operatorname{h}(p'')} p'' \alpha = \beta \alpha' p'' (\alpha \beta \alpha') p'' \alpha = p_1^2 = p^2,
\end{equation*}
which is what we wanted to show.
Similarly, $p_2q_2 \equiv p_2^2$.

($b.ii$) Finally, suppose $p,q$ do not share a common leftmost or rightmost nontrivial subpath. 
Then, since $p,q$ factor into the paths (\ref{p q factor}), we have the setup shown in Figure \ref{nguzka}.b.ii.
(Although not drawn, $p^{2+}$ and $q^{2+}$ may intersect themselves multiple times.)

Factor $p$ into arrow subpaths, $p = a_n \cdots a_2a_1$, $a_j \in Q_1$.
Denote by 
\begin{equation*}
p_j := a_{j-1} \cdots a_{j+1}a_j
\end{equation*}
the cyclic permutation of $p$ starting with arrow $a_j$.
Since $z \in Z$ is central and the relations $I$ are generated by binomials in paths of $Q$, for each $j \in [1,m]$ there are cycles $p'_j$, $q_j$ at $\operatorname{h}(a_{j-1})$ for which
\begin{multline*}
(p'_j - q_j)(a_{j-1} \cdots a_2a_1) = z (a_{j-1} \cdots a_1)\\
\equiv (a_{j-1} \cdots a_1) z = (a_{j-1} \cdots a_1) (p-q) = p_j(a_{j-1} \cdots a_1) - (a_{j-1} \cdots a_1)q.
\end{multline*}
Upon setting $p'_j (a_{j-1} \cdots a_1) = p_j (a_{j-1} \cdots a_1)$, we have
\begin{equation} \label{qa = aq}
q_j(a_{j-1} \cdots a_1) \equiv (a_{j-1} \cdots a_1)q.
\end{equation}
In particular, without loss of generality we may assume
\begin{equation} \label{wlog}
ze_{\operatorname{t}(p_j)} = p_j - q_j + I.
\end{equation}
Whence,
\begin{equation} \label{cc}
q_{j+1}a_j \equiv a_jq_j
\end{equation}
since $p_{j+1}a_j = a_j \cdots a_{j+1} a_j = a_jp_j$ and
\begin{equation*}
(p_{j+1} - q_{j+1})a_j = za_j \equiv a_jz = a_j(p_j - q_j).
\end{equation*}

Suppose $p_j \equiv q_j$ for some $j \in [1,m]$.
Then (\ref{qa = aq}) implies 
\begin{multline*}
pq = (a_m \cdots a_j)(a_{j-1} \cdots a_1)q \equiv (a_m \cdots a_j)q_j(a_{j-1} \cdots a_1)\\ \equiv (a_m \cdots a_j)p_j(a_{j-1} \cdots a_1) = p^2,
\end{multline*}
which is what we wanted to show.

So suppose $p_j \not \equiv q_j$ for each $j \in [1,m]$.
Then $p_j, q_j$ is a non-cancellative pair: $p,q$ is a non-cancellative pair, so there is a path $r$ such that $(p-q)r \equiv 0$, and therefore
\begin{equation*}
(p_j - q_j) (a_{j-1} \cdots a_1) r \equiv (a_{j-1} \cdots a_1) (p-q)r \equiv 0.
\end{equation*}
We may assume that $q_j$ is a representative of $q_j+I$ for which the region $\mathcal{R}_{p_j,q_j}$ contains a minimal number of unit cycles.

By assumption, there are minimal indices $1 \leq k < \ell \leq n$ such that $\operatorname{t}(a_k) = \operatorname{h}(a_{\ell})$.
Factor $q_j$ into paths $q_j = \beta_j q'_j$, where $\beta_j$ is the maximal leftmost subpath of $q_j$ that is a subpath of a unit cycle.
Then, by the minimality of $\mathcal{R}_{p_j,q_j}$ and the minimality of the index $k$, (\ref{cc}) implies that $a_j\beta_j$ is an arc for $j \in [1, k-1]$.
This is shown in Figure \ref{nguzka2}, where each $c_j \in Q_1$ is an arrow and $c_j a_j \beta_j$ is a unit cycle.
Observe that the complementary arc to $(a_j \beta_j)^+$ lies in the region $\mathcal{R}_{p_j,q_j}$.

Since $a_j$ is a rightmost arrow subpath of $p_j$ and $ze_{\operatorname{t}(p_j)} = p_j - q_j + I$ by (\ref{wlog}), we may assume that $a_j$ is not a rightmost arrow subpath of any representative of $q_j+I$ since otherwise we have case (b.i) with $p_j,q_j$ in place of $p,q$.
In particular, $q_j$ does not have a unit cycle subpath modulo $I$, by Lemma \ref{unit cycle lemma}. 

Now consider (\ref{cc}) with $j = k$:
\begin{equation*} \label{needs unit cycle}
a_k q_k \equiv q_{k+1}a_k.
\end{equation*}
As we have just shown, the arrow $a_k$ is not a rightmost subpath of $q_k$ modulo $I$.
But it is clear from Figure \ref{nguzka2} that $(a_k \beta_k)^+$ cannot be an arc whose complementary arc lies in $\mathcal{R}_{p_k,q_k}$.
Therefore, by the minimality of $\mathcal{R}_{p_k,q_k}$, we have $a_k q_k \not \equiv s a_k$ for all paths $s$.
Consequently $ze_{\operatorname{t}(p_k)} \not = p_k - q_k + I$, contrary to (\ref{wlog}).
\end{proof}

\begin{figure}
\begin{equation*}
\xy 0;/r.47pc/:
(-30,12.5)*+{\text{\scriptsize{$i$}}}="1a";(-35,7.5)*{\cdot}="2a";(-25,7.5)*{\cdot}="3a";(-35,2.5)*{\cdot}="4a";(-25,2.5)*{\cdot}="5a";
(-30,-12.5)*{\cdot}="10a";(-35,-7.5)*{\cdot}="8a";(-25,-7.5)*{\cdot}="9a";(-35,-2.5)*{\cdot}="6a";(-25,-2.5)*{\cdot}="7a";
(0,12.5)*+{\text{\scriptsize{$i$}}}="1b";(-5,7.5)*{\cdot}="2b";(5,7.5)*{\cdot}="3b";(-5,2.5)*{\cdot}="4b";(5,2.5)*{\cdot}="5b";
(0,-12.5)*{\cdot}="10b";(-5,-7.5)*{\cdot}="8b";(5,-7.5)*{\cdot}="9b";(-5,-2.5)*{\cdot}="6b";(5,-2.5)*{\cdot}="7b";
(-5,12.5)*{\cdot}="11b";(-10,7.5)*{\cdot}="12b";(-10,2.5)*{\cdot}="13b";(-10,-2.5)*{\cdot}="14b";(-6,-9)*{\cdot}="15b";
(-15,-9)*{}="16b";(-24,-12.5)*{\cdot}="17b";(-18,-12.5)*{\cdot}="18b";(-12,-12.5)*{\cdot}="19b";(-6,-12.5)*{\cdot}="20b";
(-10,-7.5)*{\cdot}="22b";
(30,12.5)*+{\text{\scriptsize{$i$}}}="1c";(25,7.5)*{\cdot}="2c";(35,7.5)*{\cdot}="3c";(25,2.5)*{\cdot}="4c";(35,2.5)*{\cdot}="5c";
(30,-12.5)*{\cdot}="10c";(25,-7.5)*{\cdot}="8c";(35,-7.5)*{\cdot}="9c";(25,-2.5)*{\cdot}="6c";(35,-2.5)*{\cdot}="7c";
(24,-12.5)*{\cdot}="20c";(18,-12.5)*{\cdot}="19c";(12,-12.5)*{\cdot}="18c";(6,-12.5)*{\cdot}="17c";
(-30,-5)*{\vdots}="";(-15,-5)*{\vdots}="";
{\ar@[blue]^{q'_1}"1a";"11b"};
{\ar@[blue]^{q=q_1=\beta_1 q'_1}"1b";"1c"};
{\ar@[red]|-{a_1}"1a";"2a"};{\ar@[red]|-{a_2}"2a";"4a"};{\ar@[red]|-{a_3}"4a";"6a"};{\ar@[red]"6a";"8a"};{\ar@[red]|-{a_{k-1}}"8a";"10a"};
{\ar@[red]"10a";"9a"};{\ar@[red]"9a";"7a"};{\ar@[red]"7a";"5a"};{\ar@[red]"5a";"3a"};{\ar@[red]|-{a_n}"3a";"1a"};
{\ar@[red]|-{a_1}"1b";"2b"};{\ar@[red]|-{a_2}"2b";"4b"};{\ar@[red]|-{a_3}"4b";"6b"};{\ar@[red]"6b";"8b"};{\ar@[red]"8b";"10b"};(-1.9,-8.9)*{\text{\scriptsize{$a_{k-1}$}}}="";
{\ar@[red]|-{a_{\ell+1}}"10b";"9b"};{\ar@[red]"9b";"7b"};{\ar@[red]"7b";"5b"};{\ar@[red]"5b";"3b"};{\ar@[red]|-{a_n}"3b";"1b"};
{\ar@[red]|-{a_1}"1c";"2c"};{\ar@[red]|-{a_2}"2c";"4c"};{\ar@[red]|-{a_3}"4c";"6c"};{\ar@[red]"6c";"8c"};{\ar@[red]|-{a_{k-1}}"8c";"10c"};
{\ar@[red]|-{a_{\ell+1}}"10c";"9c"};{\ar@[red]"9c";"7c"};{\ar@[red]"7c";"5c"};{\ar@[red]"5c";"3c"};{\ar@[red]|-{a_n}"3c";"1c"};
{\ar@[blue]"2a";"3a"};{\ar@[blue]"4a";"5a"};{\ar@[blue]"6a";"7a"};{\ar@[blue]"8a";"9a"};
{\ar@[blue]^{\beta_1}"11b";"1b"};{\ar@[blue]|-{\beta_2}"12b";"2b"};{\ar@[blue]|-{\beta_3}"13b";"4b"};{\ar@[blue]|-{\beta_4}"14b";"6b"};{\ar@[blue]|-{\beta_{k-1}}"22b";"8b"};
{\ar@[blue]|-{q'_2}"3a";"12b"};{\ar@[blue]|-{q'_3}"5a";"13b"};{\ar@[blue]|-{q'_4}"7a";"14b"};{\ar@[blue]|-{q'_{k-1}}"9a";"22b"};
{\ar@[blue]@{-}@/^.25pc/"18b";"16b"};{\ar@[blue]|-{q''_k}"16b";"15b"};{\ar@[blue]@/^.15pc/"15b";"10b"};(-5,-10.2)*{\text{\scriptsize{$\beta_k$}}}="";
(-18.5,-11)*{\text{\scriptsize{$q'_k$}}}="";
{\ar@[red]_{a_k}"10a";"17b"};{\ar@[red]_{\cdots}"17b";"18b"};{\ar@[red]"18b";"19b"};{\ar@[red]_{\cdots}"19b";"20b"};{\ar@[red]_{a_{\ell}}"20b";"10b"};
{\ar@[red]_{a_k}"10b";"17c"};{\ar@[red]_{\cdots}"17c";"18c"};{\ar@[red]"18c";"19c"};{\ar@[red]_{\cdots}"19c";"20c"};{\ar@[red]_{a_{\ell}}"20c";"10c"};
{\ar@[teal]|-{c_1}"2b";"11b"};{\ar@[teal]|-{c_2}"4b";"12b"};{\ar@[teal]|-{c_3}"6b";"13b"};
{\ar@{..>}@/_1pc/"13b";"6b"};{\ar@{..>}@/_1pc/"12b";"4b"};
{\ar@{..>}@/_2pc/|-{\gamma}"11b";"2b"};
\endxy
\end{equation*}
\caption{Case ($b.ii$) in the proof of Proposition \ref{Lucy}.}
\label{nguzka2}
\end{figure}

\begin{Theorem} \label{nil}
Let $A$ be a nonnoetherian dimer algebra with center $Z$ and $\psi: A \to A'$ a cyclic contraction.
Then
\begin{equation*}
Z \cap \operatorname{ker} \psi = \operatorname{nil}Z.
\end{equation*}
\end{Theorem}

\begin{proof}
(i) We first claim that if $z \in Z \cap \operatorname{ker}\psi$, then $z^2 = 0$, and in particular $z \in \operatorname{nil}Z$.

Consider a central element $z$ in $\operatorname{ker} \psi$.
Since $z$ is central it commutes with the vertex idempotents, and so $z$ is a $k$-linear combination of cycles.
Therefore, since $\psi$ sends paths to paths and $I'$ is generated by certain differences of paths, it suffices to suppose that $z$ is of the form
\begin{equation*}
z = \sum_{i \in Q_0}(p_i - q_i),
\end{equation*}
where $p_i$, $q_i$ are cycles in $e_iAe_i$ with equal $\psi$-images modulo $I'$.
Note that there may be vertices $i \in Q_0$ for which $p_i = q_i = 0$.

By Proposition \ref{Lucy}, we have
\begin{equation*}
p_iq_i = p_ip_i = q_iq_i = q_ip_i.
\end{equation*}
Therefore
\begin{equation*}
z^2 = ( \sum_{i \in Q_0} \left( p_i - q_i \right))^2 = \sum_{i \in Q_0} \left( p_i - q_i \right)^2 = 0.
\end{equation*}

(ii) We now claim that if $z \in \operatorname{nil}Z$, then $z \in \operatorname{ker} \psi$.

Suppose $z^n = 0$.
Then for each $i \in Q_0$, we have
\begin{equation*}
\bar{\tau}_{\psi}(ze_i)^n \stackrel{\textsc{(i)}}{=} \bar{\tau}_{\psi}\left( (ze_i)^n \right) \stackrel{\textsc{(ii)}}{=} \bar{\tau}_{\psi} \left( z^n e_i \right) = 0,
\end{equation*}
where (\textsc{i}) holds since $\bar{\tau}_{\psi}$ is an algebra homomorphism on $e_iAe_i$, and (\textsc{ii}) holds since $z$ is central.
But $\bar{\tau}_{\psi}\left(e_iAe_i \right)$ is contained in the integral domain $k[\mathcal{S}']$.
Whence
\begin{equation*}
\bar{\tau}(\psi(ze_i)) = \bar{\tau}_{\psi}(ze_i) = 0.
\end{equation*}
Thus $\psi(ze_i) = 0$ since $\bar{\tau}$ is injective, by Lemma \ref{from B}.2.
Therefore
\begin{equation*}
\psi(z) \stackrel{\textsc{(i)}}{=} \psi\left(z \sum_{i\in Q_0} e_i \right) \stackrel{\textsc{(ii)}}{=} \sum_{i\in Q_0} \psi(ze_i) = 0,
\end{equation*}
where (\textsc{i}) holds since the vertex idempotents form a complete set and (\textsc{ii}) holds since $\psi$ is a $k$-linear map.
\end{proof}

\section{The central nilradical is prime}

Let $Z$ and $S$ be the center and cycle algebra of a nonnoetherian dimer algebra $A = kQ/I$, let $\psi: A \to A'$ be a cyclic contraction, and let $R \subset S$ be the center of the ghor algebra $\Lambda$ of $Q$.
In this section we will show that the reduced center $\hat{Z} := Z/\operatorname{nil}Z$ of $A$ is an integral domain.

\begin{Theorem} \label{subalgebra}
There is an exact sequence of $Z$-modules
\begin{equation} \label{exact seq}
0 \longrightarrow \operatorname{nil}Z \longhookrightarrow Z \stackrel{\bar{\psi}}{\longrightarrow} R,
\end{equation}
where $\bar{\psi}$ is an algebra homomorphism.
Therefore $\hat{Z} := Z/\operatorname{nil}Z$ is isomorphic to a subalgebra of $R$.
\end{Theorem}

\begin{proof}
(i) We first claim that for each $i \in Q_0$, the map
\begin{equation} \label{exact seq1}
\bar{\psi}: Z \to R, \ \ \ z \mapsto \overbar{ze_i},
\end{equation}
is a well-defined algebra homomorphism and independent of the choice of $i$.

Consider a central element $z \in Z$ and vertices $j,k \in Q_0$.
Since $Q$ is a dimer quiver, there is a path $p$ from $j$ to $k$.
For $i \in Q_0$, set $z_i := ze_i \in e_iAe_i$.
Recall that $\bar{\tau}_{\psi}$ is an algebra homomorphism on each vertex corner ring $e_iAe_i$.
Thus
\begin{equation*}
\overbar{p}  \overbar{z}_j = \overbar{pz_j} = \overbar{pz} = \overbar{zp} = \overbar{z_kp} = \overbar{z}_k  \overbar{p} \in k[\mathcal{S}'].
\end{equation*}
But $\overbar{p} = \bar{\tau}(\psi(p))$ is nonzero since $\bar{\tau}$ is injective by Lemma \ref{from B}.2, and the $\psi$-image of any path is nonzero.
Thus, since $k[\mathcal{S}']$ is an integral domain,
\begin{equation} \label{block universe}
\overbar{z}_j = \overbar{z}_k.
\end{equation}
Therefore, since $j,k \in Q_0$ were arbitrary,
\begin{equation*}
\overbar{z}_j \in k\left[ \cap_{i \in Q_0} \bar{\tau}_{\psi}(e_iAe_i) \right] = R.
\end{equation*}

(ii) Let $z \in Z$, $i \in Q_0$, and suppose $\psi(ze_i) = 0$.
We claim that $\psi(z) = 0$.

For each $j \in Q'_0$, denote by
\begin{equation*}
c_j := \left| \psi^{-1}(j) \cap Q_0 \right|
\end{equation*}
the number of vertices in $\psi^{-1}(j)$.
Since $\psi$ maps $Q_0$ surjectively onto $Q'_0$, we have $c_j \geq 1$.
Furthermore, if $k \in \psi^{-1}(j)$, then
\begin{equation} \label{cj}
\psi(z)e_j = c_j \psi(ze_k).
\end{equation}
Set
\begin{equation*}
z'_j := c_j^{-1} \psi(z)e_j.
\end{equation*}
Then
\begin{equation*}
z' := \sum_{j \in Q'_0} z'_j
\end{equation*}
is in the center $Z'$ of $A'$ by (\ref{block universe}) and \cite[(6) and Theorem 5.9.1]{B2}.\footnote{Note that $\psi(z)$ is not in $Z'$ if there are vertices $i,j \in Q'_0$ for which $c_i \not = c_j$.
Therefore, in general $\psi(Z)$ is not contained in $Z'$.}
Whence, for each $j \in Q'_0$,
\begin{equation*}
\bar{\tau}(z'_j) \stackrel{(\textsc{i})}{=} \bar{\tau}(z'_{\psi(i)}) = \bar{\tau}( c_{\psi(i)}^{-1} \psi(z)e_{\psi(i)}) \stackrel{\textsc{(ii)}}{=} \bar{\tau}(\psi(ze_i)) = 0,
\end{equation*}
where (\textsc{i}) holds by (\ref{block universe}) and (\textsc{ii}) holds by (\ref{cj}).
Thus, each $z'_j$ vanishes since $\bar{\tau}$ is injective, by Lemma \ref{from B}.2.
Therefore
\begin{equation*}
\psi(z) = \sum_{j \in Q'_0} \psi(z)e_j = \sum_{j \in Q'_0} c_jz'_j = 0.
\end{equation*}

(iii) We now claim that the homomorphism (\ref{exact seq1}) can be extended to the exact sequence (\ref{exact seq}).
Let $z \in \operatorname{ker} \bar{\psi}$.
Then for each $i \in Q_0$,
\begin{equation*}
\bar{\tau}(\psi(ze_i)) = \overbar{ze_i} = \bar{\psi}(z) = 0.
\end{equation*}
Whence $\psi(ze_i) = 0$ since $\bar{\tau}$ is injective.
Thus $\psi(z) = 0$ by Claim (ii).
Therefore, by Theorem \ref{nil},
\begin{equation*}
z \in \operatorname{ker}\psi \cap Z = \operatorname{nil}Z.
\end{equation*}
\end{proof}

\begin{Corollary} \label{domains}
The algebras $\hat{Z}$ and $R$ are integral domains.
Therefore, the central nilradical $\operatorname{nil}Z$ of $A$ is a prime ideal of $Z$.
In particular, the schemes $\operatorname{Spec}Z$, $\operatorname{Spec}\hat{Z}$, and $\operatorname{Spec}R$ are irreducible.
\end{Corollary}

\begin{proof}
$R$ and $S$ are domains since they are subalgebras of the domain $k[\mathcal{S}']$.
Therefore $\hat{Z}$ is a domain since it isomorphic to a subalgebra of $R$ by Theorem \ref{subalgebra}.
\end{proof}

For brevity, we will identify $\hat{Z}$ with its isomorphic $\bar{\psi}$-image in $R$ (Theorem \ref{subalgebra}), and thus write $\hat{Z} \subseteq R$.

The following example shows that it is possible for the reduced center $\hat{Z}$ to be properly contained in the ghor center $R$.
However, they determine the same nonnoetherian variety \cite{B6}, and we will show below that their normalizations are equal (Theorem \ref{integral closure theorem}).

\begin{Example} \label{iso R} \rm{
Dimer algebras exist for which the containment $\hat{Z} \hookrightarrow R$ is proper.
Indeed, consider the contraction given in Figure \ref{1}.
This contraction is cyclic since the cycle algebra is preserved:
\begin{equation*}
S = k[x^2, xy, y^2, z] = S'.
\end{equation*}
We claim that the reduced center $\hat{Z}$ of $A = kQ/I$ is not isomorphic to $R$.
By the exact sequence (\ref{exact seq}), it suffices to show that the homomorphism $\bar{\psi}: Z \hookrightarrow R$ is not surjective.

We claim that the monomial $z\sigma$ is in $R$, but is not in the image $\bar{\psi}(Z)$.
It is clear that $z \sigma$ is in $R$ from the $\bar{\tau}_{\psi}$ labeling of arrows given in Figure \ref{1}.

Assume to the contrary that $z\sigma \in \bar{\psi}(Z)$.
Then, by (\ref{block universe}), for each $j \in Q_0$ there is an element in $Ze_j$ whose $\bar{\tau}_{\psi}$-image is $z\sigma$.
Consider the vertex $i \in Q_0$ shown in Figure \ref{Yoohoo}.
The set of cycles in $e_iAe_i$ with $\bar{\tau}_{\psi}$-image $z \sigma$ are drawn in red.
As is shown in the figure, none of these cycles `commute' with both of the arrows with tail at $i$.
Therefore $\hat{Z} \not \cong R$.
}\end{Example}

\begin{figure}
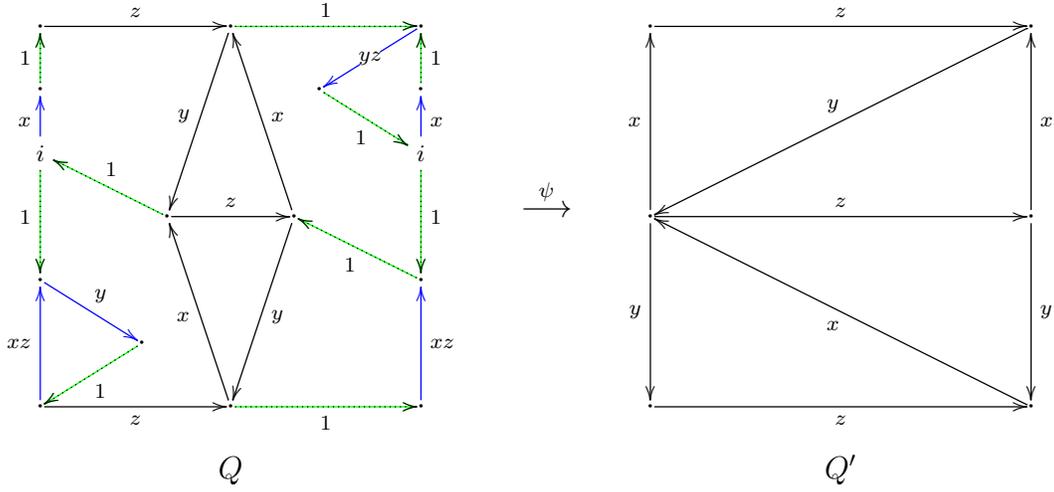

\begin{equation*}
\begin{array}{ccc}
\xy 0;/r.4pc/:
(-15,-15)*{\cdot}="1";(-15,-5)*{\cdot}="2";(-15,5)*+{\text{\scriptsize{$i$}}}="3";(-15,15)*{\cdot}="4";
(-7,-10)*{\cdot}="5";(0,-15)*{\cdot}="6";(-5,0)*{\cdot}="7";(0,15)*{\cdot}="8";
(5,0)*{\cdot}="9";(15,-15)*{\cdot}="10";(15,-5)*{\cdot}="12";
(15,5)*+{\text{\scriptsize{$i$}}}="13";(15,15)*{\cdot}="14";(7,10)*{\cdot}="15";
(-15,10)*{\cdot}="17";(15,10)*{\cdot}="18";
{\ar^{1}@[green]"5";"1"};{\ar@{..>}"5";"1"};{\ar^{y}@[blue]"2";"5"};{\ar^{xz}@[blue]"1";"2"};{\ar_{1}@[green]"3";"2"};{\ar@{..>}"3";"2"};
{\ar^{x}@[blue]"3";"17"};{\ar^{1}@[green]"17";"4"};{\ar@{..>}"17";"4"};
{\ar_{z}"1";"6"};{\ar^{x}"6";"7"};{\ar_{1}@[green]"7";"3"};{\ar@{..>}"7";"3"};
{\ar_{y}"8";"7"};{\ar^{z}"7";"9"};{\ar^{y}"9";"6"};{\ar_{x}"9";"8"};{\ar^{1}@[green]"8";"14"};{\ar@{..>}"8";"14"};
{\ar^{1}@[green]"12";"9"};{\ar@{..>}"12";"9"};{\ar_{1}@[green]"6";"10"};{\ar@{..>}"6";"10"};{\ar^{1}@[green]"13";"12"};{\ar@{..>}"13";"12"};{\ar_{xz}@[blue]"10";"12"};{\ar|-{yz}@[blue]"14";"15"};
{\ar_{1}@[green]"15";"13"};{\ar@{..>}"15";"13"};{\ar^{z}"4";"8"};
{\ar_{x}@[blue]"13";"18"};{\ar_{1}@[green]"18";"14"};{\ar@{..>}"18";"14"};
\endxy
& \ \ \ \stackrel{\psi}{\longrightarrow} \ \ \ &
\xy 0;/r.4pc/:
(-15,-15)*{\cdot}="1";(-15,0)*{\cdot}="2";(-15,15)*{\cdot}="3";
(15,-15)*{\cdot}="4";(15,0)*{\cdot}="5";(15,15)*{\cdot}="6";
{\ar_{y}"2";"1"};{\ar^{x}"2";"3"};{\ar_{z}"1";"4"};{\ar^{x}"4";"2"};{\ar^{z}"2";"5"};{\ar_{y}"6";"2"};
{\ar^{z}"3";"6"};{\ar^{y}"5";"4"};{\ar_{x}"5";"6"};
\endxy \\
Q & & Q'
\end{array}
\end{equation*}
\caption{A cyclic contraction $\psi: A \to A'$ for which the containment $\hat{Z} \hookrightarrow R$ is proper.
$Q$ and $Q'$ are drawn on a torus, and the contracted arrows are drawn in green.
The arrows drawn in blue form removable 2-cycles under $\psi$.
The arrows in $Q$ are labeled by their $\bar{\tau}_{\psi}$-images, and the arrows in $Q'$ are labeled by their $\bar{\tau}$-images.}
\label{1}
\end{figure}

\begin{figure}
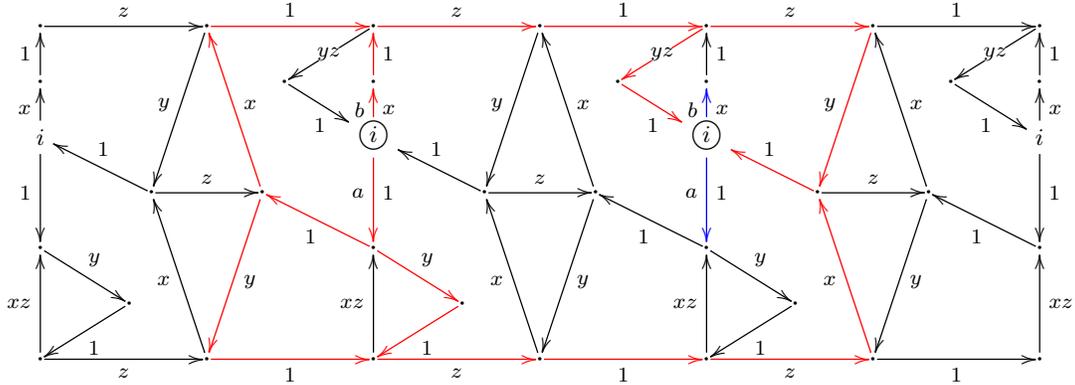

\begin{equation*}
\xy 0;/r.35pc/:
(-15,-15)*{\cdot}="1";(-15,-5)*{\cdot}="2";(-15,5)*+{\textcircled{\scriptsize{$i$}}}="3";
(-15,15)*{\cdot}="4";
(-7,-10)*{\cdot}="5";(0,-15)*{\cdot}="6";(-5,0)*{\cdot}="7";(0,15)*{\cdot}="8";
(5,0)*{\cdot}="9";(15,-15)*{\cdot}="10";(15,-5)*{\cdot}="12";
(15,5)*+{\textcircled{\scriptsize{$i$}}}="13";(15,15)*{\cdot}="14";(7,10)*{\cdot}="15";
(-15,10)*{\cdot}="17";(15,10)*{\cdot}="18";
(-45,-15)*{\cdot}="1a";(-45,-5)*{\cdot}="2a";(-45,5)*+{\text{\scriptsize{$i$}}}="3a";
(-45,15)*{\cdot}="4a";
(-37,-10)*{\cdot}="5a";(-30,-15)*{\cdot}="6a";(-35,0)*{\cdot}="7a";(-30,15)*{\cdot}="8a";
(-25,0)*{\cdot}="9a";
(-23,10)*{\cdot}="15a";
(-45,10)*{\cdot}="17a";
(23,-10)*{\cdot}="5b";(30,-15)*{\cdot}="6b";(25,0)*{\cdot}="7b";(30,15)*{\cdot}="8b";
(35,0)*{\cdot}="9b";(45,-15)*{\cdot}="10b";(45,-5)*{\cdot}="12b";
(45,5)*+{\text{\scriptsize{$i$}}}="13b";(45,15)*{\cdot}="14b";(37,10)*{\cdot}="15b";
(45,10)*{\cdot}="18b";
{\ar^{1}@[red]"5";"1"};{\ar^{y}@[red]"2";"5"};{\ar^{xz}"1";"2"};{\ar^{1}_a@[red]"3";"2"};
{\ar_{x}^b@[red]"3";"17"};{\ar_{1}@[red]"17";"4"};
{\ar_{z}@[red]"1";"6"};{\ar^{x}"6";"7"};{\ar_{1}"7";"3"};
{\ar_{y}"8";"7"};{\ar^{z}"7";"9"};{\ar^{y}"9";"6"};{\ar_{x}"9";"8"};{\ar^{1}@[red]"8";"14"};
{\ar^{1}"12";"9"};{\ar_{1}@[red]"6";"10"};{\ar^{1}_a@[blue]"13";"12"};{\ar^{xz}"10";"12"};{\ar|-{yz}@[red]"14";"15"};
{\ar_{1}@[red]"15";"13"};{\ar^{z}@[red]"4";"8"};
{\ar_{x}^b@[blue]"13";"18"};{\ar_{1}"18";"14"};
{\ar^{1}"5a";"1a"};{\ar^{y}"2a";"5a"};{\ar^{xz}"1a";"2a"};{\ar_{1}"3a";"2a"};
{\ar^{x}"3a";"17a"};{\ar^{1}"17a";"4a"};
{\ar_{z}"1a";"6a"};{\ar^{x}"6a";"7a"};{\ar_{1}"7a";"3a"};
{\ar_{y}"8a";"7a"};{\ar^{z}"7a";"9a"};{\ar^{y}@[red]"9a";"6a"};{\ar_{x}@[red]"9a";"8a"};{\ar^{1}@[red]"8a";"4"};
{\ar^{1}@[red]"2";"9a"};{\ar_{1}@[red]"6a";"1"};{\ar|-{yz}"4";"15a"};
{\ar_{1}"15a";"3"};{\ar^{z}"4a";"8a"};
{\ar^{1}"5b";"10"};{\ar^{y}"12";"5b"};
{\ar_{z}@[red]"10";"6b"};{\ar^{x}@[red]"6b";"7b"};{\ar_{1}@[red]"7b";"13"};
{\ar_{y}@[red]"8b";"7b"};{\ar^{z}"7b";"9b"};{\ar^{y}"9b";"6b"};{\ar_{x}"9b";"8b"};{\ar^{1}"8b";"14b"};
{\ar^{1}"12b";"9b"};{\ar_{1}"6b";"10b"};{\ar^{1}"13b";"12b"};{\ar_{xz}"10b";"12b"};{\ar|-{yz}"14b";"15b"};
{\ar_{1}"15b";"13b"};{\ar^{z}@[red]"14";"8b"};
{\ar_{x}"13b";"18b"};{\ar_{1}"18b";"14b"};
\endxy
\end{equation*}
\caption{There are six cycles $p_1, \ldots, p_6$ (or five distinct cycles modulo $I$) at $i$ with $\bar{\tau}_{\psi}$-image $z \sigma = xyz^2$, drawn in red.
There are two arrows with tail at $i$, labeled $a$ and $b$.
Observe that if the rightmost arrow subpath of $p_j$ is $a$ (resp.\ $b$), then $bp_j$ (resp.\ $ap_j$) cannot homotope to a path whose rightmost arrow subpath is $b$ (resp.\ $a$).
Therefore there is no cycle $q_j$ for which $bp_j \equiv q_jb$ (resp.\ $ap_j \equiv q_ja$).
Consequently, $z\sigma$ is in $R \setminus \hat{Z}$.}
\label{Yoohoo}
\end{figure}

Example \ref{iso R} raises the following question.

\begin{Question} \label{if and only if question} \rm{
Are there necessary and sufficient conditions for the isomorphism $\hat{Z} \cong R$ to hold?
}\end{Question}

\section{Normalization of the reduced center}

Let $Z$ and $S$ be the center and cycle algebra of a nonnoetherian dimer algebra $A = kQ/I$, let $\psi: A \to A'$ be a cyclic contraction, and let $R \subset S$ be the center of the ghor algebra $\Lambda$ of $Q$.
In this section we will show that the normalizations of $\hat{Z} = Z/\operatorname{nil}Z$ and $R$ are equal, nonnoetherian, and properly contained in the cycle algebra $S$.
We denote by $\bar{Z}$ and $\bar{R}$ their respective normalizations.
Recall that for $g, h \in S$, we write $g \mid h$ if $g$ divides $h$ in the polynomial ring $k[\mathcal{S}']$.

\begin{Lemma} \label{generalize}
Let $u \in \mathbb{Z}^2 \setminus 0$.
Suppose $a \in Q_1$, $p \in \hat{\mathcal{C}}^u_{\operatorname{t}(a)}$, and $q \in \hat{\mathcal{C}}^u_{\operatorname{h}(a)}$.
If $\mathcal{R}^{\circ}_{ap,qa}$ contains no vertices, then $ap = qa$.
Consequently, $\overbar{p} = \overbar{q}$.
\end{Lemma}

\begin{proof}
Suppose that there are representatives of $(ap)^+$ and $(qa)^+$ that bound a compact region $\mathcal{R}_{ap,qa}$ with no vertices in its interior.
If $(ap)^+$ and $(qa)^+$ have no cyclic subpaths (modulo $I$), then $ap = qa$ by Lemma \ref{4.13.2}.

So suppose $(qa)^+$ contains a cyclic subpath.
The path $q^+$ has no cyclic subpaths since $q$ is in $\hat{\mathcal{C}}$.
Thus $q$ factors into paths $q = q_2q_1$, where $(q_1a)^+$ is a cycle.
In particular,
\begin{equation*}
\operatorname{t}(p^+) = \operatorname{t}((q_1q_2)^+) \ \ \ \text{ and } \ \ \ \operatorname{h}(p^+) = \operatorname{h}((q_1q_2)^+).
\end{equation*}
Whence $p$ and $q_1q_2$ bound a compact region $\mathcal{R}_{p,q_1q_2}$.
Furthermore, its interior $\mathcal{R}^{\circ}_{p,q_1q_2}$ contains no vertices since $\mathcal{R}^{\circ}_{ap,qa}$ contains no vertices.

The path $(q^2)^+$ has no cyclic subpaths, again since $q$ is in $\hat{\mathcal{C}}$.
Thus $(q_1q_2)^+$ also has no cyclic subpaths.
Furthermore, $p^+$ has no cyclic subpaths since $p$ is in $\hat{\mathcal{C}}$.
Therefore $p = q_1q_2$, again by Lemma \ref{4.13.2}.

Since there are no vertices in $\mathcal{R}^{\circ}_{ap,qa}$, there are also no vertices in the interior of the region bounded by the cycle $(aq_1)^+$.
Thus, there is some $\ell \geq 1$ such that
\begin{equation*}
aq_1 = \sigma^{\ell}_{\operatorname{h}(a)} \ \ \text{ and } \ \ q_1a = \sigma^{\ell}_{\operatorname{t}(a)}.
\end{equation*}
Therefore,
\begin{equation*}
ap = aq_1q_2 = \sigma^{\ell}_{\operatorname{h}(a)}q_2 \stackrel{\textsc{(i)}}{=} q_2 \sigma^{\ell}_{\operatorname{t}(a)} = q_2q_1a = qa,
\end{equation*}
where (\textsc{i}) holds by Lemma \ref{unit cycle lemma}.

Finally, $ap = qa$ and $\overbar{a} \not = 0$ together imply $\overbar{p} = \overbar{q}$ by Lemma \ref{from B}.1.
\end{proof}

\begin{Lemma} \label{t in S}
If $g$ is a monomial in $k[\mathcal{S}']$ and $g \sigma$ is in $S$, then $g$ is also in $S$.
\end{Lemma}

\begin{proof}
Suppose $g$ is a monomial in $k[\mathcal{S}']$ for which $g \sigma$ is in $S$.
Then there is a cycle $p \in A'$ such that $\overbar{p} = g\sigma$.
Let $u \in \mathbb{Z}^2$ be such that $p \in \mathcal{C}'^u$.
Since $A'$ is cancellative, $\hat{\mathcal{C}}'^u \not = \emptyset$ by \cite[Proposition 4.11]{B2}; fix $q \in \hat{\mathcal{C}}'^u$. 
Then $\sigma \nmid \overbar{q}$ by \cite[Proposition 4.21.1]{B2}.
Thus, there is some $m \geq 1$ such that
\begin{equation*}
\overbar{q} \sigma^m = \overbar{p} = g \sigma,
\end{equation*}
by \cite[Lemma 4.19]{B2}.
Therefore,
\begin{equation*}
g = (g \sigma) \sigma^{-1} = \overbar{q} \sigma^{m-1} \in S' \stackrel{\textsc{(i)}}{=} S,
\end{equation*}
where (\textsc{i}) holds since $\psi$ is cyclic.
\end{proof}

\begin{Lemma} \label{big enough}
There is some $n \geq 1$ for which
\begin{equation} \label{nested}
\sigma^{n-1} S \not \subseteq R \ \ \text{ and } \ \ \sigma^n S \subset R.
\end{equation}
\end{Lemma}

\begin{proof}
Let $s \in S$.
Since $S$ is generated over $k$ by a set of monomials in $k[\mathcal{S}']$, we may assume that $s$ is a monomial.
In particular, there is a cycle $p$ for which $\overbar{p} = s$.

By Lemma \ref{unit cycle lemma}, there is some $n \geq 1$ such that for each $i \in Q_0$, the unit cycle $\sigma_i^n$ is equal (modulo $I$) to a cycle $q_i$ that passes through each vertex of $Q$.
Thus, the concatenated cycle $q_{\operatorname{t}(p)}p$ passes through each vertex of $Q$.
Whence $\sigma^n \overbar{p} = \overbar{q_{\operatorname{t}(p)}p}$ is in $R$.
Therefore $\sigma^n S \subset R$.
Since $S \not \subseteq R$, there is a minimal such $n \geq 1$.
\end{proof}

\begin{Proposition} \label{r in R and} \
\begin{enumerate}
 \item If $r \in R$ and $\sigma \nmid r$, then $r \in \hat{Z}$.
 \item If $s \in S$, then there is some $n \geq 0$ such that for each $m \geq 1$, $s^m \sigma^n \in \hat{Z}$.
 \item If $r \in R$, then there is some $n \geq 1$ such that $r^n \in \hat{Z}$.
\end{enumerate}
\end{Proposition}

\begin{proof}
(1) Since $R$ is generated over $k$ by a set of monomials in $k[\mathcal{S}']$, it suffices to consider a monomial $r \in R$ for which $\sigma \nmid r$.
For each vertex $i \in Q_0$, then, there is a cycle $c_i \in e_iAe_i$ satisfying $\overbar{c}_i = r$. 
Fix $a \in Q_1$ and set
\begin{equation*}
p := c_{\operatorname{t}(a)} \ \ \ \text{ and } \ \ \ q := c_{\operatorname{h}(a)}.
\end{equation*}
See Figure \ref{presentism?}.
We claim that $ap = qa$.

Let $u, v \in \mathbb{Z}^2$ be such that
\begin{equation*}
p \in \mathcal{C}^u \ \ \ \text{ and } \ \ \ q \in \mathcal{C}^v.
\end{equation*}
Then $u = v$ since $\overbar{p} = r = \overbar{q}$, by Lemma \ref{cyclelemma}.2.
Furthermore, $u \not = 0$ since $\sigma \nmid r$, by Lemma \ref{cyclelemma}.1. 
Therefore, $(ap)^+$ and $(qa)^+$ bound a compact region $\mathcal{R}_{ap,qa}$ in $\mathbb{R}^2$.

We proceed by induction on the number of vertices in the interior $\mathcal{R}^{\circ}_{ap,qa}$.

First suppose there are no vertices in $\mathcal{R}^{\circ}_{ap,qa}$.
Since $\sigma \nmid r = \overbar{p} = \overbar{q}$, the cycles $p$ and $q$ are in $\hat{\mathcal{C}}$, by Lemma \ref{cyclelemma}.3. 
Therefore $ap = qa$, by Lemma \ref{generalize}.

So suppose $\mathcal{R}^{\circ}_{ap,qa}$ contains at least one vertex $i^+$.
Let $w \in \mathbb{Z}^2$ be such that $c_i \in \mathcal{C}^w$.
Then $w = u = v$, again by Lemma \ref{cyclelemma}.2.
Therefore $c_i$ intersects $p$ at least twice or $q$ at least twice.
Suppose $c_i$ intersects $p$ at vertices $j$ and $k$.
Then $p$ factors into paths
\begin{equation*}
p = p_2 e_k b e_j p_1 = p_2bp_1.
\end{equation*}
Let $d^+$ be the subpath of $(c_i^2)^+$ from $j^+$ to $k^+$.
Then
\begin{equation*}
\operatorname{t}( d^+ ) = \operatorname{t}( b^+ ) = j^+ \ \ \ \text{ and } \ \ \ \operatorname{h}( d^+ ) = \operatorname{h}( b^+) = k^+.
\end{equation*}
In particular, $d^+$ and $b^+$ bound a compact region $\mathcal{R}_{d,b}$.

Since we are free to choose the vertex $i^+$ in $\mathcal{R}^{\circ}_{ap,qa}$, we may suppose $\mathcal{R}^{\circ}_{d,b}$ contains no vertices.
Furthermore, $c_i^+$ and $p^+$ have no cyclic subpaths since $\sigma \nmid r$, by Lemma \ref{cyclelemma}.1.
Thus, their respective subpaths $d^+$ and $b^+$ have no cyclic subpaths.
Whence $d = b$, by Lemma \ref{4.13.2}.

Furthermore, since $\mathcal{R}^{\circ}_{ap_2dp_1,qa}$ contains less vertices than $\mathcal{R}^{\circ}_{ap,qa}$, it follows by induction that
\begin{equation*}
a p_2dp_1 = qa.
\end{equation*}
Therefore
\begin{equation*}
ap = a (p_2bp_1) = a (p_2dp_1) = qa,
\end{equation*}
proving our claim.

Finally, since $a \in Q_1$ was arbitrary, the sum $\sum_{i \in Q_0} c_i$ is central in $A$.

(2) Let $s \in S$ be a monomial.
By Lemma \ref{big enough}, there is an $N \geq 0$ such that $s^m \sigma^N$ is in $R$ for each $m \geq 1$.
Fix $m \geq 1$ and set $r: = s^m \sigma^N$.
Then, for each $i \in Q_0$, there is a cycle $c_i \in e_i A e_i$ for which $\overbar{c}_i = r$.

Fix an arrow $a \in Q_1$.
Set $i := \operatorname{t}(a)$ and $j := \operatorname{h}(a)$.
Let $t^+$ be a path in $Q^+$ from $\operatorname{h}( (ac_i)^+)$ to $\operatorname{t}( (ac_i)^+)$.
Then by Lemma \ref{p sigma = q sigma}, there is some $\ell, m_i, m_j \geq 0$ such that
\begin{equation*}
tc_ja\sigma_i^{\ell} = \sigma_j^{m_i} \ \ \ \text{ and } \ \ \ ac_i t \sigma_j^{\ell} = \sigma_j^{m_j}.
\end{equation*}
Thus,
\begin{equation*}
\sigma^{m_i} = \bar{\tau}_{\psi}(tc_ja\sigma_i^{\ell}) = \overbar{t}  \overbar{c}_j \overbar{a} \sigma^{\ell}
= \overbar{a}  \overbar{c}_i \overbar{t} \sigma^{\ell}
= \bar{\tau}_{\psi}(ac_it \sigma_j^{\ell}) = \sigma^{m_j}.
\end{equation*}
Furthermore, $\sigma \not = 1$ since $\bar{\tau}$ is injective by Lemma \ref{from B}.2.
Whence $m := m_i = m_j$ since $k[\mathcal{S}']$ is an integral domain.
Therefore,
\begin{equation} \label{aci}
a \left( c_i \sigma_i^{m} \right) = a c_i\left(tc_j a \sigma_i^{\ell} \right) \stackrel{\textsc{(i)}}{=} \left( ac_it\sigma_j^{\ell} \right)c_ja = \sigma_j^{m} c_ja
\stackrel{\textsc{(ii)}}{=} \left(c_j \sigma_j^{m} \right) a,
\end{equation}
where (\textsc{i}) and (\textsc{ii}) hold by Lemma \ref{unit cycle lemma}.

For each $a \in Q_1$ there is an $m = m(a)$ such that (\ref{aci}) holds.
Set
\begin{equation*}
n := \operatorname{max} \{ m(a) \ | \ a \in Q_1 \}.
\end{equation*}
Then (\ref{aci}) implies that the element $\sum_{i \in Q_0} c_i \sigma_i^n$ is central.
Furthermore, for each $k \in Q_0$,
\begin{equation*}
\bar{\tau}_{\psi}(c_k \sigma_k^n) = r\sigma^n = s^m \sigma^{N + n}.
\end{equation*}
The claim then follows since $m \geq 1$ was arbitrary. 

(3) By Claim (1), it suffices to suppose $\sigma \mid r$.
Then there is a monomial or scalar $g \in k[\mathcal{S}']$ such that $r = g \sigma$.
By Lemma \ref{t in S}, $g$ is in $S$ since $g \sigma = r \in R \subset S$.
Therefore, by Claim (2), there is some $n \geq 1$ such that
\begin{equation*}
r^n = g^n \sigma^n \in \hat{Z}.
\end{equation*}
\end{proof}

\begin{figure}
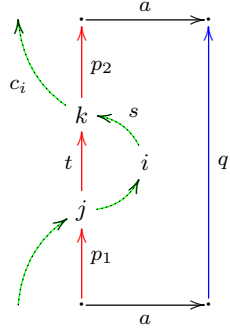

\begin{equation*}
\xy 0;/r.5pc/:
(-6,-9)*{}="1";(-2,-9)*{\cdot}="2";(6,-9)*{\cdot}="3";
(-2,-3)*+{\text{\scriptsize{$j$}}}="4";(2,0)*+{\text{\scriptsize{$i$}}}="5";
(-2,3)*+{\text{\scriptsize{$k$}}}="6";
(-6,9)*{}="7";(-2,9)*{\cdot}="8";(6,9)*{\cdot}="9";
{\ar_{a}"2";"3"};{\ar^{a}"8";"9"};
{\ar_{p_1}@[red]"2";"4"};{\ar^{b}@[red]"4";"6"};{\ar_{p_2}@[red]"6";"8"};
{\ar_{q}@[blue]"3";"9"};
{\ar@/^/^{}@[teal]"1";"4"};{\ar@/_/@[teal]"4";"5"};
{\ar@/_/_{d}@[teal]"5";"6"};{\ar@/^/^{c_i}@[teal]"6";"7"};
\endxy
\end{equation*}
\caption{Setup for Proposition \ref{r in R and}.1.
The cycles $p = p_2tp_1$, $q$, $c_i$, are drawn in red, blue, and teal respectively.
The path $a$ is an arrow.}
\label{presentism?}
\end{figure}

\begin{Theorem} \label{integral closure theorem}
The normalizations of the reduced and ghor centers are isomorphic,
\begin{equation*}
\bar{Z} \cong \bar{R}.
\end{equation*}
\end{Theorem}

\begin{proof}
By \cite[Theorem 1.1]{B6}, the fraction fields of $\hat{Z}$, $R$, and $S$ coincide, 
\begin{equation} \label{FracZ}
\operatorname{Frac}\hat{Z} = \operatorname{Frac}R = \operatorname{Frac}S.
\end{equation}
Thus the inclusion $\hat{Z} \subseteq R$ implies
\begin{equation} \label{Z red subseteq}
\bar{Z} \subseteq \bar{R}.
\end{equation}

To show $\bar{Z} \supseteq \bar{R}$, consider $r \in R$.
Then there is some $n \geq 1$ such that $r^n \in \hat{Z}$, by Proposition \ref{r in R and}.3.
Whence, $r$ is a root of the monic polynomial
\begin{equation*}
x^n - r^n \in \hat{Z}[x].
\end{equation*}
Thus $r$ is in $\bar{Z}$.
Therefore
\begin{equation} \label{R subseteq overline}
R \subseteq \bar{Z}.
\end{equation}
But then
\begin{equation*}
\bar{R} \stackrel{\textsc{(i)}}{\subseteq} \bar{Z} \stackrel{\textsc{(ii)}}{\subseteq} \bar{R},
\end{equation*}
where (\textsc{i}) holds by (\ref{R subseteq overline}) and (\textsc{ii}) holds by (\ref{Z red subseteq}).
Therefore $\bar{R} \cong \bar{Z}$.
\end{proof}

\begin{Proposition} \label{ngu3}
If $s \in S \setminus R$ is a monomial and $\sigma \nmid s$, then $s^n$ is not in $R$ for each $n \geq 1$.
Moreover, there exists such a monomial $s$.
\end{Proposition}

\begin{proof}
Holds by \cite[Proposition 3.14]{B4} since $A$ is nonnoetherian.
\end{proof}

\begin{Proposition} \label{ngu4}
Let $s \in S \setminus R$ be a monomial.
Then $\sigma \mid s$ if and only if $s \in \bar{R}$. 
\end{Proposition}

\begin{proof}
First suppose ${\sigma \mid s}$. 
Then $s = g \sigma$ for some monomial $g \in k[\mathcal{S}']$.
Whence $g$ is in $S$, by Lemma \ref{t in S}.
Thus for sufficiently large $n \geq 1$, $(g \sigma)^n$ is in $R$, by Lemma \ref{big enough}. 
But then $x^n - (g\sigma)^n$ is in $R[x]$.
Furthermore, 
\begin{equation*}
g \sigma \in S \subset \operatorname{Frac}S = \operatorname{Frac}R,
\end{equation*}
where the last equality holds by (\ref{FracZ}). 
Consequently, $s = g \sigma$ is in the normalization $\bar{R}$.

Now suppose ${\sigma \nmid s}$. 
Then $s^n$ is in not in $R$ for each $n \geq 1$, by Proposition \ref{ngu3}.
Furthermore, since $R$ is generated by monomials in a polynomial ring (in particular, $R$ is toric), each monomial in $\bar{R}$ is a root of a monic binomial $x^n - r \in R[x]$ for some $n \geq 1$ and $r \in R$. 
But then $s$ is not in $\bar{R}$ since $s^n$ is not in $R$ for each $n \geq 1$.
\end{proof}

\begin{Theorem} \label{nonnoetherian theorem}
The normalizations $\bar{R} \cong \bar{Z}$ are nonnoetherian and properly contained in the cycle algebra $S$.
\end{Theorem}

\begin{proof}
Since $A$ is nonnoetherian, there is a monomial $s \in S \setminus R$ for which $\sigma \nmid s$, by Proposition \ref{ngu3}. 
Thus $s$ is not in $\bar{R}$, by Proposition \ref{ngu4}.
Whence $\bar{R} \not = S$.
But $S = S' = R'$ is the center of the noetherian (equivalently, cancellative) dimer algebra $A'$ and so is normal.
Consequently, the inclusion $R \subset S$ implies $\bar{R} \subseteq \bar{S} = S$.
Therefore $\bar{R}$ is properly contained in $S$.

Moreover, since $\sigma = \prod_{x \in \mathcal{S}'}x$, we also have $\sigma \nmid s^n$ for each $n \geq 1$.
Thus, $s^n$ is not in $\bar{R}$ for each $n \geq 1$, again by Proposition \ref{ngu4}.
It follows that $\bar{R}$ is nonnoetherian by \cite[Claims (i) and (iii) in the proof of Theorem 3.16]{B4}. 

Alternatively, recall that an element $s$ of the fraction field $\operatorname{Frac}T$ of an integral domain $T$ is said to be `almost integral over $T$' if there is some nonzero $t \in T$ such that $s^m t$ is in $T$ for all $m \geq 0$. 
It is well known that if $T$ is noetherian, then every almost integral element $s \in \operatorname{Frac}T$ over $T$ is integral over $T$ \cite[Lemma 10.37.4]{S}.

Now let $s \in S \setminus \bar{R}$ be as above with $\sigma \nmid s$.
Then $s \in \operatorname{Frac}S = \operatorname{Frac}R = \operatorname{Frac}\bar{R}$, by (\ref{FracZ}). 
Thus, for $n \geq 1$ sufficiently large, $s^m \sigma^n \in R \subseteq \bar{R}$ for all $m \geq 0$, by Lemma \ref{big enough}.
Whence, $s$ is almost integral over both $R$ and $\bar{R}$.
But $s \not \in \bar{R}$, so $s$ is not integral over $R$ or $\bar{R}$.
Therefore both $R$ and $\bar{R}$ are nonnoetherian.
\end{proof}

\section{Three characterizations of normality}

Let $Z$ and $S$ be the center and cycle algebra of a nonnoetherian dimer algebra $A = kQ/I$, let $\psi: A \to A'$ be a cyclic contraction, and let $R \subset S$ be the center of the ghor algebra $\Lambda$ of $Q$.
In this section we will introduce three equivalent conditions for $R$ to be normal.
These conditions provide an explicit description of the reduced center $\hat{Z} = Z/\operatorname{nil}Z$ if it is normal.

\begin{Lemma} \label{st not in R}
Suppose $r \in R$ and $s \in S$ are monomials.
If $r \not = \sigma^n$ for any $n \geq 1$, then $rs \in R$.
\end{Lemma}

\begin{proof}
Suppose the hypotheses hold. 
Fix $i \in Q_0$.
Since $r$ is a monomial in $R$ and $s$ is a monomial in $S$, there are $u,v \in \mathbb{Z}^2$ and cycles
\begin{equation*}
p \in \mathcal{C}_i^u \ \ \ \text{ and } \ \ \ q \in \mathcal{C}^v
\end{equation*}
such that
\begin{equation*}
\overbar{p} = r \ \ \ \text{ and } \ \ \ \overbar{q} = s.
\end{equation*}

The assumption $r \not = \sigma^n$ for $n \geq 1$ implies $u \not = 0$, by Lemma \ref{cyclelemma}.1.
Furthermore, if $v = 0$, then $\overbar{q} = \sigma^m$ for some $m \geq 1$, again by Lemma \ref{cyclelemma}.1. 
But then $s = \overbar{q}$ is in $R$ since $\sigma$ is in $R$.
Whence, $rs$ is in $R$.

We may thus suppose that $u, v$ are both nonzero.

(i) If $v = u$, then $\overbar{p} = \overbar{q}\sigma^m$ for some $m \in \mathbb{Z}$, by Lemma \ref{cyclelemma}.2. 

(i.a) If $m \leq 0$, then ${s = r\sigma^{|m|}}$ is in $R$ since $r$ and $\sigma$ are both in $R$.
Whence, ${rs = r^2\sigma^{|m|}}$ is in $R$.

(i.b) So suppose $m > 0$; then ${r = s \sigma^{|m|}}$.
In particular, ${\sigma \mid r = \overbar{p}}$.
Thus, $p$ is not in $\hat{\mathcal{C}}$, by Lemma \ref{cyclelemma}.3.
Consequently, there is a cyclic subpath $c^+$ of the lift $p^{2+}$.
It suffices to suppose that $p$ factors into paths 
\begin{equation*} \label{p = s2p's1}
p = d_2p'd_1
\end{equation*}
with $c = d_1d_2 \in \mathcal{C}^0$ and $p' \in \hat{\mathcal{C}}$; otherwise, if $p'$ is not in $\hat{\mathcal{C}}$, then repeat the argument with $p'$ in place of $p$.
Since $p'$ is in $\hat{\mathcal{C}}$, we have $\sigma \nmid \overbar{p}'$, by Lemma \ref{cyclelemma}.3.
Furthermore, since $c$ is in $\mathcal{C}^0$, we have $p' \in \mathcal{C}^u$.
Thus $\overbar{p} = \overbar{p}'\sigma^{\ell}$ for some $\ell \in \mathbb{Z}$, by Lemma \ref{cyclelemma}.2.
Whence, 
\begin{equation*}
\overbar{p}' \sigma^{\ell} = \overbar{p} = r = s\sigma^{m}.
\end{equation*}

Without loss of generality we may assume ${\sigma \nmid s}$.
Then, since {$\sigma \nmid \overbar{p}'$}, we have $\ell = m$ and $\overbar{p}' = s$.
Therefore the cycle $d_2p'^2d_1$ has tail at $i$ and $\bar{\tau}_{\psi}$-image $s^2\sigma^m$.
Since $i \in Q_0$ was arbitrary, it follows that $rs = s^2\sigma^m$ is in $R$.

(ii) Finally, suppose $v \not = u$.
Then, since $u,v \not = 0$ (and the surface is a torus), the lifts $p^+$ and $q^+$ intersect at some vertex $j^+$ in $Q^+$. 
Consequently, $p$ and $q$ factor into paths
\begin{equation*}
p = p_2e_jp_1 \ \ \ \text{ and } \ \ \ q = q_2e_jq_1.
\end{equation*}
We may thus form the cycle
\begin{equation*}
c = p_2q_1q_2p_1 \in e_iAe_i
\end{equation*}
with $\bar{\tau}_{\psi}$-image
\begin{equation*}
\overbar{c} = \overbar{p}  \overbar{q} = rs.
\end{equation*}
But $i \in Q_0$ was arbitrary, and so it again follows that $\overbar{c} = rs$ is in $R$. 
\end{proof}

\begin{Proposition} \label{R normal}
A ghor center $R$ is normal if and only if $\sigma S \subset R$.
\end{Proposition}

\begin{proof}
(1) First suppose $\sigma S \subset R$.

It is well known that cancellative dimer algebras (on a torus) are noncommutative crepant resolutions, and in particular that their centers are normal domains (e.g., \cite{Br,D}).
Moreover, $A'$ is cancellative and its center is isomorphic to $S$ \cite[Theorem 1.1.3]{B2}.
Thus, $S$ is a normal domain. 
Therefore, since $R$ is a subalgebra of $S$, we have
\begin{equation} \label{R in S normal}
\bar{R} \subseteq S.
\end{equation}

Now let $s \in S \setminus R$.
We claim that $s$ is not in $\bar{R}$.
Indeed, assume otherwise.
Since $S$ is generated by monomials in the polynomial ring $k[\mathcal{S}']$, there are monomials $s_1, \ldots, s_{\ell} \in S$ and scalars $s_0, c_1, \ldots, c_{\ell} \in k$ such that
\begin{equation*}
s = s_0 + c_1s_1 + \cdots + c_{\ell}s_{\ell}.
\end{equation*}
Since $s \not \in R$, there is some $1 \leq k \leq \ell$ such that $s_k \not \in R$.
Choose $s_k$ to have maximal degree among the subset of monomials in $\{ s_1, \ldots, s_{\ell} \}$ which are not in $R$.

If ${\sigma \mid s_k}$ in $k[\mathcal{S}']$, then there would be a monomial $g \in k[\mathcal{S}']$ such that $s_k = \sigma g$.
Furthermore, $g$ would be in $S$ by Lemma \ref{t in S}.
Whence, $s_k = \sigma g$ would be in $R$ by our assumption that $\sigma S \subset R$.
But this is not possible since $s_k$ is not in $R$.
Therefore
\begin{equation*}
\sigma \nmid s_k.
\end{equation*}

By assumption $s$ is in $\bar{R}$, and so there is some $n \geq 1$ and $r_0, \ldots, r_{n-1} \in R$ for which
\begin{equation} \label{sn + r}
s^n + r_{n-1}s^{n-1} + \cdots + r_1s = -r_0 \in R.
\end{equation}
The summand $s_k^n$ of $s^n$ is not in $R$ since $\sigma \nmid s_k$ \cite[Proposition 3.14]{B4}.
Thus $-s_k^n$ is a summand of the left-hand side of (\ref{sn + r}).
In particular, for some $1 \leq m \leq n$, there are monomial or scalar summands $r'$ of $r_m$ and $s' = s_{j_1} \cdots s_{j_m}$ of $s^m$, and a nonzero scalar $c \in k$, such that
\begin{equation*}
r's' = c s_k^n.
\end{equation*}

Since $\sigma \nmid s_k$, we have $\sigma \nmid s_k^n$, and thus $\sigma \nmid r'$.
Whence $r' \not = \sigma^m$ for any $m \geq 1$.
Thus, $r'$ is a nonzero scalar since $r' \in R$, $s' \in S$, and $s_k^n \not \in R$, by Lemma \ref{st not in R}. 
Therefore
\begin{equation} \label{impossible}
s_{j_1} \cdots s_{j_m} = s' = (c/r') s_k^n.
\end{equation}
Consequently, $s_{j_1}, \ldots, s_{j_m}$ is not in $R$.
But $m \leq n$ and the monomial $s_k$ was chosen to have maximal degree, and so (\ref{impossible}) is not possible.
Hence,
\begin{equation} \label{R cap S}
\bar{R} \cap S \subseteq R.
\end{equation}

It follows from (\ref{R in S normal}) and (\ref{R cap S}) that
\begin{equation*}
\bar{R} = \bar{R} \cap S \subseteq R \subseteq \bar{R}.
\end{equation*}
Therefore $R = \bar{R}$ is normal.

(2) Now suppose $\sigma S \not \subset R$.
Then there is a monomial $s \in S \setminus R$ for which $\sigma \mid s$.
Furthermore, $s$ is in $\bar{R}$ by Proposition \ref{ngu4}. 
Consequently, $s$ is in $\bar{R} \setminus R$ and so $R \not = \bar{R}$. 
\end{proof}

\begin{Corollary} \
\begin{enumerate}
 \item If the head or tail of each contracted arrow has indegree 1, then $R$ is normal.
 \item If $\psi$ contracts precisely one arrow, then $R$ is normal.
\end{enumerate}
\end{Corollary}

\begin{proof}
In both cases (1) and (2), clearly $\sigma S \subset R$.
\end{proof}

\begin{Proposition} \label{R is not normal}
For each $n \geq 1$, there exist ghor algebras for which (\ref{nested}) holds.
Consequently, there are ghor algebras whose centers are not normal.
\end{Proposition}

\begin{proof}
Recall the conifold quiver $Q$ with one nested square given in Figure \ref{holy smokes batman}.i.
Clearly $\sigma S \subset R$.
More generally, the conifold quiver with $n \geq 1$ nested squares satisfies (\ref{nested}); see Figure \ref{normal figure}.
The corresponding ghor center $R$ is therefore not normal for $n \geq 2$ by Proposition \ref{R normal}.
\end{proof}

\begin{figure}
\begin{equation*}
\begin{array}{ccc}
\xy 0;/r.43pc/:
(-12,24)*+{\text{\scriptsize{$1$}}}="1";(12,24)*+{\text{\scriptsize{$2$}}}="2";
(-12,0)*+{\text{\scriptsize{$2$}}}="3";(12,0)*+{\text{\scriptsize{$1$}}}="4";
(-12,-24)*+{\text{\scriptsize{$1$}}}="5";(12,-24)*+{\text{\scriptsize{$2$}}}="6";
(-9,21)*{\cdot}="7";(9,21)*{\cdot}="8";(-9,3)*{\cdot}="9";(9,3)*{\cdot}="10";
{\ar"1";"2"};{\ar@[red]|-a"2";"4"};{\ar@[red]|-b"4";"3"};{\ar"3";"1"};{\ar"3";"5"};{\ar"5";"6"};{\ar"6";"4"};
{\ar"9";"3"};{\ar"3";"10"};{\ar"10";"4"};{\ar"4";"8"};{\ar"8";"2"};{\ar"2";"7"};{\ar"7";"1"};{\ar"1";"9"};
{\ar"9";"7"};{\ar"7";"8"};{\ar"8";"10"};{\ar"10";"9"};
\endxy
&
\xy 0;/r.43pc/:
(-12,24)*+{\text{\scriptsize{$1$}}}="1";(12,24)*+{\text{\scriptsize{$2$}}}="2";
(-12,0)*+{\text{\scriptsize{$2$}}}="3";(12,0)*+{\text{\scriptsize{$1$}}}="4";
(-12,-24)*+{\text{\scriptsize{$1$}}}="5";(12,-24)*+{\text{\scriptsize{$2$}}}="6";
(-9,21)*{\cdot}="7";(9,21)*{\cdot}="8";(-9,3)*{\cdot}="9";(9,3)*{\cdot}="10";
(-6,18)*{\cdot}="11";(6,18)*{\cdot}="12";(-6,6)*{\cdot}="13";(6,6)*{\cdot}="14";
{\ar"1";"2"};{\ar@[red]|-a"2";"4"};{\ar@[red]|-b"4";"3"};{\ar"3";"1"};{\ar"3";"5"};{\ar"5";"6"};{\ar"6";"4"};
{\ar"9";"3"};{\ar"3";"10"};{\ar"10";"4"};{\ar"4";"8"};{\ar"8";"2"};{\ar"2";"7"};{\ar"7";"1"};{\ar"1";"9"};
{\ar"9";"7"};{\ar"7";"8"};{\ar"8";"10"};{\ar"10";"9"};
{\ar"13";"9"};{\ar"9";"14"};{\ar"14";"10"};{\ar"10";"12"};{\ar"12";"8"};{\ar"8";"11"};{\ar"11";"7"};{\ar"7";"13"};
{\ar"13";"11"};{\ar"11";"12"};{\ar"12";"14"};{\ar"14";"13"};
\endxy
&
\xy 0;/r.43pc/:
(-12,24)*+{\text{\scriptsize{$1$}}}="1";(12,24)*+{\text{\scriptsize{$2$}}}="2";
(-12,0)*+{\text{\scriptsize{$2$}}}="3";(12,0)*+{\text{\scriptsize{$1$}}}="4";
(-12,-24)*+{\text{\scriptsize{$1$}}}="5";(12,-24)*+{\text{\scriptsize{$2$}}}="6";
(-9,21)*{\cdot}="7";(9,21)*{\cdot}="8";(-9,3)*{\cdot}="9";(9,3)*{\cdot}="10";
(-6,18)*{\cdot}="11";(6,18)*{\cdot}="12";(-6,6)*{\cdot}="13";(6,6)*{\cdot}="14";
(-3,15)*{\cdot}="15";(3,15)*{\cdot}="16";(-3,9)*{\cdot}="17";(3,9)*{\cdot}="18";
{\ar"1";"2"};{\ar@[red]|-a"2";"4"};{\ar@[red]|-b"4";"3"};{\ar"3";"1"};{\ar"3";"5"};{\ar"5";"6"};{\ar"6";"4"};
{\ar"9";"3"};{\ar"3";"10"};{\ar"10";"4"};{\ar"4";"8"};{\ar"8";"2"};{\ar"2";"7"};{\ar"7";"1"};{\ar"1";"9"};
{\ar"9";"7"};{\ar"7";"8"};{\ar"8";"10"};{\ar"10";"9"};
{\ar"13";"9"};{\ar"9";"14"};{\ar"14";"10"};{\ar"10";"12"};{\ar"12";"8"};{\ar"8";"11"};{\ar"11";"7"};{\ar"7";"13"};
{\ar"13";"11"};{\ar"11";"12"};{\ar"12";"14"};{\ar"14";"13"};
{\ar"17";"13"};{\ar"13";"18"};{\ar"18";"14"};{\ar"14";"16"};{\ar"16";"12"};{\ar"12";"15"};{\ar"15";"11"};{\ar"11";"17"};
{\ar"17";"15"};{\ar"15";"16"};{\ar"16";"18"};{\ar"18";"17"};
\endxy
\\
(i) & (ii) & (iii)
\end{array}
\end{equation*}
\caption{Examples for Proposition \ref{R is not normal}.
Each quiver is drawn on a torus.
In each case, set $p := ba$ and let $n \geq 0$ be the minimum for which $\overbar{p} \sigma^n$ is in the ghor center $R$.
In (i) we have $n = 1$; (ii) $n = 2$; and (iii) $n = 3$.
More generally, these values yield $\sigma^n S \subset R$ and $\sigma^{n-1}S \not \subset R$.
Consequently, only the ghor center of (i) is normal; the ghor centers of (ii) and (iii) are not normal.}
\label{normal figure}
\end{figure}

Let $\mathfrak{m}_0 \in \operatorname{Max}R$ be the maximal ideal generated by all monomials in $R$.
Let $\tilde{\mathfrak{m}}_0 \subset \mathfrak{m}_0$ be the ideal of $R$ generated by all monomials in $R$ which are not powers of $\sigma$.
Then
\begin{equation*}
\mathfrak{m}_0 = (\tilde{\mathfrak{m}}_0, \sigma)R.
\end{equation*}

\begin{Proposition} \label{R = k etc} \
\begin{enumerate}
 \item $\tilde{\mathfrak{m}}_0 = \tilde{\mathfrak{m}}_0S$, and thus $\tilde{\mathfrak{m}}_0$ is an ideal of both $R$ and $S$.
 \item Let $n \geq 1$, and suppose $\sigma^{n} S \subset R$.
Then
\begin{equation*} \label{k sigma}
R = k[\sigma] + (\tilde{\mathfrak{m}}_0, \sigma^{n})S.
\end{equation*}
\end{enumerate}
\end{Proposition}

\begin{proof}
The equality $\tilde{\mathfrak{m}}_0 = \tilde{\mathfrak{m}}_0S$ follows from Lemma \ref{st not in R}.
Thus,
\begin{equation*}
R \stackrel{\textsc{(i)}}{\subseteq} k[\sigma] + \tilde{\mathfrak{m}}_0 \subseteq  k[\sigma] + (\tilde{\mathfrak{m}}_0, \sigma^{n})S \stackrel{\textsc{(ii)}}{\subseteq} R,
\end{equation*}
where (\textsc{i}) holds since $R$ is generated by monomials and (\textsc{ii}) holds since ${\tilde{\mathfrak{m}}_0S = \tilde{\mathfrak{m}}_0 \subset R}$.
\end{proof}

\begin{Theorem} \label{R normal'}
Let $R$ and $S$ be the center and cycle algebra of a ghor algebra.
The following are equivalent:
\begin{enumerate}
 \item $R$ is normal.
 \item $\sigma S \subset R$.
 \item $R = k + \mathfrak{m}_0S$.
 \item $R = k + J$ for some ideal $J$ in $S$.
\end{enumerate}
\end{Theorem}

\begin{proof}
If the ghor algebra is noetherian, then the conditions trivially hold since in this case $R = S$ and $R$ is normal.
So suppose the ghor algebra is nonnoetherian.

\indent (1) $\Leftrightarrow$ (2) holds by Proposition \ref{R normal}.

(2) $\Rightarrow$ (3): Suppose $\sigma S \subset R$.
Then
\begin{equation*}
R \stackrel{\textsc{(i)}}{=} k + \mathfrak{m}_0 \subseteq k + \mathfrak{m}_0S \stackrel{\textsc{(ii)}}{\subseteq} R,
\end{equation*}
where (\textsc{i}) holds since $R$ is generated over $k$ by a set of monomials in $k[\mathcal{S}']$, and $\mathfrak{m}_0 \subset R$ is generated over $R$ by all monomials in $R$.
To show (\textsc{ii}), let $r \in \mathfrak{m}_0$ and $s \in S$; we claim that $rs \in R$.
Since $\mathfrak{m}_0$ is generated by the monomials in $R$, we may assume that $r$ is a monomial.
Thus, if $r \not = \sigma^m$ for all $m \geq 1$, then $rs \in R$ by Lemma \ref{st not in R}.
Otherwise $rs \in \sigma S$.
But $\sigma S \subset R$ by assumption, and so $rs \in R$, proving our claim.
Therefore $R = k + \mathfrak{m}_0S$.

(3) $\Rightarrow$ (2): Suppose $R = k + \mathfrak{m}_0S$.
Then, since $\sigma \in \mathfrak{m}_0$, we have $\sigma S \subset \mathfrak{m}_0S \subset R$.

(3) $\Rightarrow$ (4): Clear.

(4) $\Rightarrow$ (2): Suppose $R = k + J$ for some ideal $J$ of $S$.
By Lemma \ref{big enough}, there is some $n \geq 1$ such that $\sigma^{n-1}S \not \subseteq R$ and $\sigma^n S \subset R$.
Fix $g \in S$ for which $g\sigma^{n-1} \not \in R$.

Since $\sigma \in R = k + J$, there is some $c \in k$ such that $c+ \sigma \in J$.
Then
\begin{equation*}
c g \sigma^{n-1} = (c + \sigma)g \sigma^{n-1} - g \sigma^n \in JS + R = J + R = R.
\end{equation*}
Whence $c = 0$ since $g \sigma^{n-1} \not \in R$.
Thus $\sigma \in J$, and therefore $\sigma S \subset JS = J \subset R$.
\end{proof}

\begin{Corollary} \label{normal corollary}
If the reduced center $\hat{Z} = Z/\operatorname{nil}Z$ of a (noetherian or nonnoetherian) dimer algebra is normal, then $\hat{Z} = k + \mathfrak{m}_0S$.
\end{Corollary}

\begin{proof}
First observe that $\hat{Z} = k + \mathfrak{m}_0S$ holds trivially if the dimer algebra is noetherian: in this case, $\hat{Z} = Z = S$ and $\mathfrak{m}_0 \subset Z$ is the maximal ideal generated by all monomials in $Z$.

So suppose the dimer algebra is nonnoetherian.
If $\hat{Z}$ is normal, then
\begin{equation*}
R \subseteq \bar{R} \stackrel{\textsc{(i)}}{=} \bar{Z} = \hat{Z} \stackrel{\textsc{(ii)}}{\subseteq} R,
\end{equation*}
where (\textsc{i}) holds by Theorem \ref{integral closure theorem} and (\textsc{ii}) holds by Theorem \ref{subalgebra}.
Thus $R = \hat{Z}$, and so $R$ is normal.
But then
\begin{equation*}
\hat{Z} = R \stackrel{\textsc{(i)}}{=} k + \mathfrak{m}_0S,
\end{equation*}
where (\textsc{i}) holds by Theorem \ref{R normal'}.
\end{proof}

\ \\
\textbf{Acknowledgments.}
The author was supported by the Austrian Science Fund (FWF) grant P34854. 
Part of this article is based on work supported by the Heilbronn Institute for Mathematical Research.

\bibliographystyle{hep}
\def\cprime{$'$} \def\cprime{$'$}

\end{document}